\documentclass{article}
\usepackage[utf8]{inputenc}
\usepackage[english]{babel}
\usepackage[a4paper, total={6in, 8in}]{geometry}
\usepackage{amsmath,amsfonts,amssymb,amsthm,mathtools,bbm,bm}
\usepackage{appendix}
\usepackage{bigints}
\usepackage{cite}
\usepackage{graphicx}
\usepackage{blkarray}
\usepackage{tikz}
\usepackage{cleveref}
\crefname{equation}{}{}

\usepackage{placeins}
\usepackage[margin=0.9375cm,format=hang, font={small,it} , labelfont=normalfont,bf]{caption}
\usetikzlibrary{automata, positioning}
\usetikzlibrary{shapes.geometric, arrows, shapes.multipart}
\tikzstyle{decision} = [rectangle, rounded corners, minimum width=0.7cm, minimum height=0.8cm,text centered, draw=black, fill=red!30]
\tikzstyle{standard} = [rectangle, rounded corners, minimum width=1cm, minimum height=0.8cm,text centered, draw=black, fill=red!10]
\tikzstyle{frozen} = [rectangle, rounded corners, minimum width=4cm, minimum height=0.8cm,text centered, draw=black, fill=red!10]

\makeatletter

\let\c@table\c@figure
\makeatother 

\numberwithin{equation}{section}
\numberwithin{figure}{section}
\numberwithin{table}{section}

\theoremstyle{plain}
\newtheorem{thm}{Theorem}[section]
\theoremstyle{definition}
\newtheorem{defn}[thm]{Definition}
\theoremstyle{definition}
\newtheorem{ass}[thm]{Assumption}
\theoremstyle{plain}
\newtheorem{prob}[thm]{Problem}
\theoremstyle{plain}
\newtheorem{prop}[thm]{Proposition}
\theoremstyle{plain}
\newtheorem{cor}[thm]{Corollary}
\theoremstyle{plain}
\newtheorem{lemma}[thm]{Lemma}
\theoremstyle{plain}
\newtheorem{rem}[thm]{Remark}
\theoremstyle{definition}

\Crefname{ass}{Assumption}{Assumptions}

\newcommand{\R}{\mathbb{R}}

\newcommand{\E}{\mathbb{E}}
\newcommand{\N}{\mathbb{N}}
\newcommand{\PP}{\mathbb{P}}

\newcommand{\pr}[1]{{#1}^{\prime}}

\DeclareMathOperator*{\argmax}{arg\,max}

\DeclarePairedDelimiter\floor{\lfloor}{\rfloor}

\title{Markov decision processes with observation costs: framework and computation with a penalty scheme}
\author{Christoph Reisinger\thanks{Mathematical Institute, University of Oxford, Oxford, OX2 6GG, UK \newline\hspace*{1.8em}(\texttt{christoph.reisinger@maths.ox.ac.uk}, \texttt{jonathan.tam@maths.ox.ac.uk})} \and Jonathan Tam\footnotemark[1]}
\date{December 5, 2023}

\begin{document}

\maketitle

\abstract{We consider Markov decision processes where the state of the chain is only given at chosen observation times and of a cost. Optimal strategies involve the optimisation of observation times as well as the subsequent action values. We consider the finite horizon and discounted infinite horizon problems, as well as an extension with parameter uncertainty. By including the time elapsed from observations as part of the augmented Markov system, the value function satisfies a system of quasi-variational inequalities (QVIs). Such a class of QVIs can be seen as an extension to the interconnected obstacle problem. We prove a comparison principle for this class of QVIs, which implies uniqueness of solutions to our proposed problem. Penalty methods are then utilised to obtain arbitrarily accurate solutions. Finally, we perform numerical experiments on three applications which illustrate our framework.}

\section{Introduction}

In this article, we examine the observation cost model (OCM) for Markov Decision Processes (MDPs). A cost must be paid in order to observe the state of the underlying MDP, and only then can adjustments be made to the action which influences the dynamics of the MDP. We propose a penalty scheme for efficient numerical computation for the resulting system of equations.

MDPs are mathematical tools that model the optimisation of a random process, in order to maximise the expected profit over time. Applications are common in maintenance, portfolio optimisation, sensor detection, reinforcement learning and more. Most setups implicitly assume a fixed source of information upon which the user relies to select an optimal action. However, such a steady stream of information might not be available in situations where resources are constrained, either by the expensive cost of measurements, or by the impracticality of frequent sampling. This calls for an extra layer of optimisation, where the user has to decide on the optimal observation times of the information source, as well as the optimal sequence of actions to maximise the expected returns.

The literature involving observation control appear across several different fields, and appear under terms such as `optimal inspections', `costly observations' or `controllable observations'. To the best of our knowledge, the earliest works appear in \cite{kushner1964optimum, meier1967optimal}, which concerns the linear quadratic Gaussian (LQG) problem over a finite horizon with fixed number of measurements, as well as the papers \cite{anderson-friedman-I, anderson-friedman-II}, which examines a costly optimal stopping problem in continuous time. Numerous applications have emerged in the literature over the years, which we list (non-exhaustively) below, broadly categorising into the following areas:
\begin{itemize}
    \item environmental management control models \cite{YOSHIOKA_1_impulse_control, Yoshioka_2_river_impulse, Yoshioka_3_monitoring_deviation, YOSHIOKA_4_biological_stopping},
    \item optimal sampling rates in communications  \cite{paging_registration, sampling_guo_2021},
    \item optimal sensing problems \cite{sensor_energy_harvesting, wu2008optimal,tzoumas2020lqg},
    \item medical treatment cycles \cite{winkelmann_markov_2014,winkelmann_phd},
    \item detection of drift in Brownian motion \cite{ekstrom2020disorder, bayraktar2015quickest, dalang2015quickest, dyrssen-ekstrom-2018},
    \item empirical works in reinforcement learning \cite{bellinger2020active, krueger2020active, bellinger2022}.
\end{itemize}

The standard approach is to formulate the problem in terms of a partially observable Markov decision process (POMDP). Solving the POMDP is then equivalent to solving a fully observable MDP on the belief state \cite{hernandez1989adaptive}. Dynamic programming for the value function leads to a search for the optimal observation time after the currently observed state, as well as the optimal action sequence between the observation times. In this article we will restrict ourselves to consider only \textit{constant} actions between observations. Such an assumption allows us to parametrise the belief state, which takes values on the space of probability measures, with a finite number of variables, so that the augmented state space becomes countable. This assumption in our framework can be relaxed to include a finitely parametrised set of time-inhomogeneous open-loop controls, as we demonstrate in \Cref{sub:inhomogeneous}. The general case, however, suffers from the curse of dimensionality: as the time between observations is unbounded, the number of actions to be optimised also grows unbounded. Indeed, non-constant controls between observations are mostly only treated under the LQG framework \cite{wu2008optimal, tzoumas2020lqg, cooper1971optimal}. For applications, the constant action assumption applies to models where actions cannot be feasibly changed without an accompanying observation, such as the medical treatment applications in \cite{winkelmann_markov_2014,winkelmann_phd} or the environmental management control models \cite{YOSHIOKA_1_impulse_control, Yoshioka_2_river_impulse, Yoshioka_3_monitoring_deviation, YOSHIOKA_4_biological_stopping}.

For the characterisation of the belief MDP for the OCM, our construction includes the time elapsed (since the last observation) as variable within dynamic programming to obtain a low-dimensional Markovian structure. To our knowledge, only the works of \cite{anderson-friedman-I, anderson-friedman-II} and \cite{paging_registration} model the OCM in this specific formulation, but the problems considered were restricted to fixed dynamics for the underlying Markov chain. The inclusion of time elapsed as a variable for the augmented Markovian system leads to a system of discrete quasi-variational inequalities (QVI). We assume that the Markov chain takes values in a finite state space $\mathcal{X}$ and that its dynamics are known and are given by the transition matrices $\{P_a\}_{a \in A}$, where $A$ is a finite action set. We also assume that the actions can only be adjusted at the observation times. The one-step reward function is given by $r(x,a) = r_{a,x}$ and the observation cost is given by a constant $c_{\mathrm{obs}} >0$. The inclusion of time elapsed as a variable in the Markov system leads to a system of discrete quasi-variational inequalities (QVI), which for the discounted infinite horizon problem, reads:
\begin{align} \label{qvi_intro}
    \min \bigg\{ & v^{n}_{a,x} - \gamma v^{n+1}_{a,x} - \Big(P^{n}_a\ r_a\Big)_x\ , \nonumber \\
    & v^{n}_{a,x} - \bigg( P^{n}_a\ \overline{\gamma v^{1} + r}\bigg)_x  +  c_{\mathrm{obs}}   \bigg\} = 0, 
\end{align}
where $v$ is the value function, indexed by: $x \in \mathcal{X}$, the state of the chain at the previous observation; $n \in \N_{\geq 1}$, the time elapsed since the previous observation; and $a \in A$, the action applied at the previous observation. The vector $\left(\overline{\gamma v^{1} + r}\right)_x = \max_{a \in A} (\gamma v^1_{a,x} + r_{a,x})$ represents the `inner loop' optimisation over the space of actions after an observation is made. 

We note here that our formulation differs from  \cite{huang2021self}, which uses the term `self-triggered MDPs' to refer to the OCM with constant action between observations. There, the time elapsed variable is not considered as part of the augmented Markov system: in our framework we parametrise the belief state by the variables $(n,x,a)$ as in \eqref{qvi_intro}; in \cite{huang2021self} the parametrisation is only with $(x,a)$. We can interpret our formulation as a special case of unknown initial conditions in a filtering problem, where actions might not have been applied optimally in the past. Such situations can occur in important settings with external factors and
constrained resources — e.g., long NHS (National Health Service in the UK) waiting times
especially after the pandemic — which clearly lead to suboptimal times of scans (observations)
and treatment (actions). Suboptimal actions can also arise due to users' general inattentiveness and inertia to reacting quickly to information \cite{reis2006consumer, reis2006producer}.

More generally, we can write the QVI \eqref{qvi_intro} obtained from dynamic programming in the following abstract form.
\begin{prob}\label{prob:intro}
\sloppy Find a function $u: \N \times \mathcal{X} \to \R^{A}$, with $u_a(\cdot,\cdot)$ as its components, satisfying
\begin{align}\label{abstract_qvi_intro}
    \min \left\{ F_a(n,x,u(n,x),u_a),\ u_a(n,x) - \mathcal{M} u(n,x) \right\}=0, \quad  n\in\N,\ x \in \mathcal{X},\ a \in A,
\end{align}
where $F_a: \N \times \mathcal{X} \times \R^{A} \times \R^{\N\times \mathcal{X}} \to \R$ satisfies \Cref{ass_intro}, and ${\mathcal{M}u: \N \times \mathcal{X} \to \R}$ is defined for all
$(n,x) \in \N \times \mathcal{X}$ by 
    \begin{align*}
        \mathcal{M}u(n,x) = \left(Q_n \overline{u^{1} - c}\right)_x ,\quad \left(\overline{u^{1} - c}\right)_x = \max_{a \in A} \left(u^1_{a,x} - c_{a,x}\right),
    \end{align*}
for a given vector $c \in \R^{\mathcal{X} \times A}$, and $\{Q_n\} \subset \mathbb{R}^{\mathcal{X}\times \mathcal{X}}$ is a sequence of substochastic matrices, i.e., with non negative row sums bounded by 1. 
\end{prob}

\begin{ass}\label{ass_intro}
The functions $\{F_a\}_{a \in A}$ in Problem \ref{prob:intro} satisfy:
    \begin{enumerate}
        \item There exists $\beta > 0$ such that for any $n \in \N$, $x \in \mathcal{X}$, bounded functions $u, v: \N \times \mathcal{X} \to \R^A$ such that $u \leq v$, and $r, s \in \R^A$ such that $\theta \coloneqq r_a - s_a = \max_{j \in A} (r_j - s_j) \geq 0$, then 
        \begin{align*}
            F_a(n, x, r, u_a + \theta) - F_a(n, x, s, v_a) \geq \beta \theta.
        \end{align*}
        \item For any $n \in N$, $a\in A$, $r \in \R^A$ and continuous bounded function $\phi: \N \times \mathcal{X} \to \R$, the functions $x \mapsto F_a(n, x, r, \phi)$ are continuous and bounded. Moreover, the functions $r \mapsto F_a(n,x, r, \phi)$ are uniformly continuous for bounded $r$, uniformly with respect to $(n, x)\in \N \times \mathcal{X}$.
    \end{enumerate}
\end{ass}

The QVI \eqref{abstract_qvi_intro} is a generalisation of a monotone system with interconnected obstacles \cite{Reisinger_Zhang_penalty}, which can arise from the discretisation of optimal switching problems in continuous time. In our case, we shall refer to the operator $\mathcal{M}$ as the inspection operator. Much like the systems with interconnected obstacles, the QVI \eqref{abstract_qvi_intro} is typically not amenable to the use of policy iteration, as the matrices arising from the inspection operator do not necessarily satisfy the M-matrix or weakly chain diagonally dominant conditions \cite{weakly_chained}. We propose instead a penalty scheme, which sees use on variational inequalities \cite{forsyth_american_penalty, power_penalty, parabolic_vi} and extensions to HJB VIs \cite{reisinger_zhang_hjbvi, Witte_Reisinger_penalty_obstacle, Witte_Reisinger_penalty}/ QVIs \cite{Reisinger_Zhang_penalty}, as an approximation. Penalty schemes have seen comparable computational performance to policy iteration in HJB QVIs, and is robust to the choice of initial estimates \cite{Witte_Reisinger_penalty_obstacle, Witte_Reisinger_penalty}. An adaptation of the penalty scheme to the QVI \eqref{abstract_qvi_intro} circumvents issues with numerical instabilities arising from computing iterates of the policy update, and the penalised equation can be solved with semismooth Newton methods. We demonstrate in \Cref{section:numerical example} that the penalty method achieves quick convergence within a few iterations on a large system whilst also mapping out accurately the optimal policy.

The main contributions of our paper are as follows:
\begin{itemize}
    \item We formulate the observation cost model (OCM) for Markov decision processes where the time elapsed after an observation is considered as part of the augmented Markov system. We present the optimality equations obtained from dynamic programming for the finite horizon problem, discounted infinite horizon problem, and the respective problems with parameter uncertainty. In all cases the optimality equations are in the form of a QVI, which are structurally different to the Bellman-type equations from existing approaches in the literature.
    \item We establish a comparison principle for the class of QVIs \eqref{abstract_qvi_intro}, of which the solution to the OCM belongs to. The class of QVIs are a generalisation of monotone systems with interconnected obstacles as seen in \cite{Reisinger_Zhang_penalty}. We propose a penalty scheme for this class of QVIs \eqref{abstract_qvi_intro}, and demonstrate the monotone convergence of the penalised solutions towards the solutions of said QVI, thereby establishing constructively the existence of solutions.
    \item We demonstrate the numerical performance of our model by applying it to the time-discretised version of the HIV-treatment model \cite{winkelmann_markov_2014}. Our framework is compatible with the original results, and also shows qualitatively different optimal behaviour when dealing with large observation gaps.
\end{itemize}

The remainder of the paper is organised as follows. \Cref{section_mc_model} sets out the framework for the OCM and establishes the corresponding set of discrete QVIs. A model problem with an explicit solution is also provided to illustrate the setup. We also outline the case of parameter uncertainty in \Cref{section_parameter_uncertainty}. In \Cref{section_comparison} we prove a comparison principle for a class of discrete QVIs which subsumes the QVIs obtained in \Cref{section_mc_model}, as well as outlining the penalty method as a numerical scheme for the QVI. Finally the numerical experiments are presented in detail in \Cref{section:numerical example}.

\subsection{Notation for MDPs and POMDPs}\label{sub:intro_notation}

As the goal of the next section is to layout the OCM precisely by formulating the model in terms of a POMDP, we shall quote here some of the standard notation for MDPs, POMDPs and a brief overview of its construction. These are largely taken from \cite[Ch 4]{hernandez1989adaptive} and we refer the reader to the references within for further detail.

We will generally be considering Markov decision processes on finite state spaces. Whilst most arguments extend naturally to more general state spaces, we shall focus on the finite setting here to retain a simplified presentation. Let $\mathcal{P}(\mathcal{X})$ denote the space of probability measures over a set $\mathcal{X}$. If $\mathcal{X}$ is finite, we will also identify $\mathcal{P}(\mathcal{X})$ with the simplex $\Delta_{\mathcal{X}}$.

\begin{defn}\label{defn:mdp}
A Markov control model is a tuple $\langle \mathcal{X}, A, p, r \rangle$, where
\begin{itemize}
    \item $\mathcal{X}$ is the finite \textit{underlying state space};
    \item $A$ is the finite \textit{action space};
    \item $p:\mathcal{X} \times A \to \Delta_{\mathcal{X}}$ is the \textit{transition kernel};
    \item $r: \mathcal{X} \times A \to \R$ is the \textit{one-step reward function}.
\end{itemize}
\end{defn}
At each time $t$, a state $x_t \in \mathcal{X}$ is observed. The controller chooses an action $a_t \in A$ and receives a reward $r(x_t, a_t)$. The system then moves to a new state $x_{t+1} \in \mathcal{X}$ with probability $p(x_{t+1} \vert x_t, a_t)$ and the process repeats at time $t+1$. Actions are chosen according to a \textit{policy} $\pi = (\pi_t)_t$, which is a sequence of kernels $\pi_t: H_t \to A$, where $H_0 \coloneqq \mathcal{X}$ and $H_t \coloneqq (\mathcal{X} \times A)^t \times \mathcal{X}$ for $t \geq 1$, known as the \textit{history set} at time $t$. The set of all policies is denoted by $\Pi$. Given an initial state $x_0 \in \mathcal{X}$ and policy $\pi \in \Pi$, by the Ionescu-Tulcea theorem (see \cite[Appendix C]{hernandez1989adaptive}), there exists a unique probability measure $\PP^{\pi}_{x_0}$ on the canonical sample space $\Omega \coloneqq H_{\infty} \coloneqq (\mathcal{X} \times A)^{\infty}$, such that given $\omega = (x_0, a_0, x_1, a_1, \ldots) \in \Omega$,
\begin{align*}
    \PP^{\pi}_{x_0}(\omega) = \delta(x_0)\ \pi_0(a_0 \mid x_0)\ p(x_1 \mid x_0, a_0)\ \pi(a_1 \mid x_0, a_0, x_1) \ldots
\end{align*}
The objective is to maximise an objective function over the set of policies $\Pi$, for example, in the finite horizon case,
\begin{align*}
    J(\pi, x_0) = \E^{\pi}_{x_0} \left[ \sum^N_{n=0} r(x_n, a_n) \right],\quad \pi \in \Pi,\ x_0 \in \mathcal{X},
\end{align*}
where $N \in \N$ is the time horizon, and $\E^{\pi}_{x_0}$ is the expectation under the measure $\PP^{\pi}_{x_0}$. The value function is given by
\begin{align*}
    v(x) = \sup_{\pi \in \Pi}J(\pi, x),\quad x \in \mathcal{X}.
\end{align*}
It is well known that an optimal policy $\pi^{*}$ for an MDP is deterministic and Markovian, i.e. there exists deterministic functions $\{\phi_t\}_{t \geq 0}$ such that for $h_t = (x_0, a_0, \ldots, x_t) \in H_t$, $\pi^{*}_t (h_t) = \phi_t(x_t)$, and $v(x) = J(\pi^{*}, x)$.

In many cases, rather than having full information of the MDP, one instead has access to noisy observations correlated to the underlying MDP. This gives rise to the notion of partially observable Markov decision processes (POMDPs), which can also be described by a given tuple as follows.

\begin{defn}\label{defn:pomdp}
    A partially observable control model is a tuple $\langle \mathcal{X}, \mathcal{O}, A, p, p_0, q, q_0, r \rangle$, where
    \begin{itemize}
        \item $\mathcal{X}$ is the finite \textit{state space};
        \item $\mathcal{O}$ is the finite \textit{observation space};
        \item $A$ is the finite \textit{action space};
        \item $p: \mathcal{X} \times A \to \Delta_{\mathcal{X}}$ is the \textit{transition kernel};
        \item $p_0 \in \Delta_{\mathcal{X}}$ is the \textit{initial distribution};
        \item $q: A \times \mathcal{X} \to \Delta_{\mathcal{O}}$ is the\textit{ observation kernel};
        \item $q_0: \mathcal{X} \to \Delta_{\mathcal{O}}$ is the \textit{initial observation kernel};
        \item $r: \mathcal{X} \times A \to [0, \infty)$ is the \textit{one-step reward function}.
    \end{itemize}
\end{defn}

In this setting, given an underlying state $x_t \in \mathcal{X}$, an observation $\bar{x}_t \in \mathcal{O}$ is generated according to the observation kernel $q(\bar{x}_t\vert a_{t-1}, x_t)$. The controller chooses an action $a_t \in A$ based on their observations, rather than the values of the underlying states. For this, define the \textit{observable history sets}
\begin{align*}
    \mathcal{H}_0 \coloneqq \mathcal{O}, \ \mathcal{H}_t \coloneqq (\mathcal{O} \times A)^t \times \mathcal{O},\quad t \geq 1.
\end{align*}
A policy for a POMDP is now a sequence of kernels $\pi_t: \mathcal{H}_t \to A$. Denote the set of policies for the POMDP as $\Pi_{\mathrm{po}}$. By the Ionescu-Tulcea theorem again, given an initial distribution $p_0 \in \Delta_{\mathcal{X}}$ and policy $\pi \in \Pi_{\mathrm{po}}$, there exists a unique probability measure $\PP^{\pi}_{p_0}$ on the canonical space $\Omega = (\mathcal{X} \times \mathcal{O} \times A)^{\infty}$ such that for $\omega = (x_0, \bar{x}_0, a_0, x_1, \bar{x}_1, a_1, \ldots) \in \Omega$,
\begin{align*}
    \PP^{\pi}_{p_0}(\omega) = p_0(x_0)\ q_0(\bar{x}_0 \mid x_0)\ \pi_0(a_0 \mid \bar{x}_0)\ p(x_1 \mid x_0, a_0)\ q(\bar{x}_1 \mid x_1, a_0)\ \pi(x_1 \mid \bar{x}_0, a_0, \bar{x}_1) \ldots
\end{align*}
The maximisation is now performed over $\pi \in \Pi_{\mathrm{po}}$, with the objective and value function
\begin{align*}
    J(\pi, p_0) = \E^{\pi}_{p_0} \left[ \sum^N_{n=0} r(x_n, a_n) \right],\ v(p) = \sup_{\pi \in \Pi_{\mathrm{po}}} J(\pi, p),\quad \pi \in \Pi_{\mathrm{po}},\ p, p_0 \in \Delta_{\mathcal{X}}.
\end{align*}
Without knowledge of the underlying states, the POMDP is a non-Markovian problem. The standard approach to solve a POMDP is to consider an equivalent (fully observable) problem, known as the \textit{belief MDP}, on the space $Z \coloneqq \Delta_{\mathcal{X}}$. The Markovian strucutre is recovered when lifted to the belief MDP, so that classical dynamic programming techniques can be applied. The transition kernel for the \textit{belief state} $z = (z_t)_t$ can be constructed as follows: given the POMDP, construct a kernel $R:Z \times A \to \mathcal{X} \times \mathcal{O}$ such that
\begin{align*}
    R(x, \bar{x} \mid z, a) = q(\bar{x} \mid a, \pr{x}) p(\pr{x} \mid x, a) z(x),\quad (x, \bar{x}) \in \mathcal{X} \times \mathcal{O},\ (z,a) \in Z \times A.
\end{align*}
It can be shown that there exists a kernel $\pr{H}: Z \times A  \to \mathcal{X}$, such that $R$ can be disintegrated into
\begin{align*}
    R(x, \bar{x} \mid z, a) = \pr{H}(x \mid z, a , \bar{x}) \pr{R}( \bar{x} \mid z, a),\quad (x, \bar{x}) \in \mathcal{X} \times \mathcal{O},\ (z,a) \in Z \times A.
\end{align*}
where $\pr{R}$ is the marginal of $R$ on $\mathcal{O}$. Then, letting $H_{z,a,\bar{x}} = \pr{H}( \cdot \mid z, a, \bar{x}) \in Z$, define the kernel $\pr{q}: Z \times A \to Z$ by
\begin{align*}
    \pr{q}(\pr{z} \mid z ,a) = \sum_{\bar{x} \in \mathcal{O}} \delta_{H_{z,a,\bar{x}}}(\pr{z}) \pr{R}(\bar{x} \mid z,a ),\quad z, \pr{z} \in Z,\ a \in A.
\end{align*}
Then one takes $\pr{q}$ as the transition kernel for the belief state, and construct the initial kernel $\pr{q}_0$ analogously. The belief MDP is then $\langle Z, A, \pr{q}, \pr{q}_0, r_z \rangle$, where $r_z : Z \times A \to \R$ as $r_z(z,a) = \sum_{x \in \mathcal{X}} r(x,a) z(x)$. The belief state can be interpreted as the conditional distribution of the underlying state $x_t$, given the observed history $(\bar{x}_0, a_0, \ldots, \bar{x}_t)$. Let $\Pi^z$ be the set of policies for the belief MDP, which are now a sequence of kernels on $A$, given the history sets $\mathcal{H}^z_t = (Z \times A)^t \times Z$. The objective and value function for the belief MDP are given by
\begin{align*}
    J_z(\pi^z, z_0) \coloneqq \E^{\pi^z}_{z_0} \left[ \sum^N_{n=0} r_z(z_n, a_n) \right],\ v_z(z) = \sup_{\pi^z \in \Pi^z} J_z(\pi^z, z),\quad \pi^z \in \Pi^z,\ z, z_0 \in Z,
\end{align*}
where $\E^{\pi^z}_{z_0}$ is the expectation over the canonical space $(Z \times A)^{\infty}$ under the policy $\pi^z$ and initial condition $z_0 \in Z$. It can then be shown that policies in $\Pi_{\mathrm{po}}$ are equivalent to policies in $\Pi^z$, in the sense that any $h_t \in \mathcal{H}_t$ can be mapped to a corresponding $h^z_t \in \mathcal{H}^z_t$, so that given $\pi^z \in \Pi^z$, one can construct a corresponding $\pi \in \Pi_{\mathrm{po}}$ via
\begin{align*}
    \pi(\cdot \mid h_t) \coloneqq \pi^z(\cdot \mid h^z_t),
\end{align*}
and the conditional probabilities assigned on the action set $A$ are the same. Moreover, $\pi^{*} \in  \Pi_{\mathrm{po}}$ is optimal for $J$ if and only if $\pi^{*,z} \in \Pi^z$ is optimal for $J_z$, and it holds that
\begin{align*}
    v_z(z) = J_z(\pi^{*,z}, z) = J(\pi^{*},z) = v(z),\quad z \in Z = \Delta_{\mathcal{X}}.
\end{align*}
Thus, when considering a POMDP, it is sufficient to consider its equivalent MDP in the belief state, of which the optimal policy is Markovian with respect to $z=(z_t)_t$.

\section{Problem formulation}\label{section_mc_model}

In the Observation Cost Model, the process evolves sequentially as follows. At each time $t$, the controller decides if they would like to pay \textit{an observation cost} $c_{\mathrm{obs}}>0$ to observe the state $x_t \in \mathcal{X}$. If they decide to do so, then the controller applies the action $a_t \in A$ according to some suitable policy $\pi$, and receives a reward $r(x_t, a_t)$. If they decide not to observe, then no cost is paid, but the controller cannot change the action value, so that $a_t= a_{t-1}$. We assume that the reward $r(x_t, a_t)$ is collected and `locked in' at time $t$, but is not observable to the user if $x_t$ is not observable. In both instances, the system moves to a new state $x_{t+1} \in \mathcal{X}$ according to the transition kernel $p(x_{t+1} \mid x_t, a_t)$.

We now formally write down the objective function, and shall make precise the terms appearing within in the rest of this section. In view of the description above, the controller wishes to maximise (for example, in the finite horizon case)
\begin{align}\label{eq:formal_ocm_objective}
    \E^{\pi}\left[ \sum^{N}_{n=0} \left( r(x_n, a_n) -  i_n \cdot c_{\mathrm{obs}} \right)\right],
\end{align}
where $\pi$ a suitably admissible control policy to be made precise later. The sequence $i=(i_n)_n$ takes values in $\mathcal{I} \coloneqq \{0, 1\}$ and will be referred to as the \textit{inspection values}. A value of $i_n = 1$ represents an observation made at time $n$, so that observation cost $c_{\mathrm{obs}}$ is deducted from the total reward in \eqref{eq:formal_ocm_objective}. Conversely no observations are made if $i_n = 0$. \Cref{flow_chart_obs} illustrates the sequential flow of a standard MDP, compared to that of the OCM.

\begin{figure}[!t]
\centering
\begin{tikzpicture}[node distance=2cm]
\draw [very thick, ->] (0,-1) -- ++(1.0,0) node[right, standard, label = $x_n$] (u1) {State};
\draw [very thick, ->] (u1.east) -- ++(1.0,0) node[right, standard, label = $a_n$] (a1) {Action};
\draw [very thick, ->] (a1.east) -- ++(1.0,0) node[right, standard, label = $x_{n+1}$] (u2) {State};
\draw [very thick, ->] (u2.east) -- ++(1.0,0) node[right, standard, label = $a_{n+1}$] (a2) {Action};
\draw [very thick, ->] (a2.east) -- ++(1.0,0);
\end{tikzpicture}\vspace{1cm}
\begin{tikzpicture}[node distance=2cm]
\draw [very thick, ->] (1,-1) -- ++(1.0,-1.0) node[right, decision] (d1) {\begin{tabular}{c}Observation \\ Decision\end{tabular}};
\draw [very thick, ->] (1,-3) -- ++(1.0,1.0);
\draw [very thick, ->] (d1.east) -- node[above left, pos = 0.6]{$i_n=1$} ++(1.0, 1.0)node[right, standard, label = $x_n$] (u1) {State};
\draw [very thick, ->] (u1.east) -- ++(1.65,0) node[right, standard, label = $a_n$] (a1) {Action};
\draw [very thick, ->] (d1.east) -- node[below left, pos = 0.6]{$i_n=0$} ++(1.0, -1.0) node[right, frozen, label= {$a_n = a_{n-1}$}] (f1) {No action change};
\draw [very thick, ->] (a1.east) -- ++(1.0,-1.0) node[right, decision] (d2) {\begin{tabular}{c}Observation \\ Decision\end{tabular}};
\draw [very thick, ->] (f1.east) -- ++(1.0, 1.0);
\draw [very thick, ->] (d2.east) -- node[above left, pos = 0.6]{$i_{n+1}=1$} ++(1.0, 1.0);
\draw [very thick, ->] (d2.east) -- node[below left, pos = 0.6]{$i_{n+1}=0$} ++(1.0, -1.0);
\end{tikzpicture}
\caption{Top: Standard formulation with full observations; Bottom: The inclusion of observation costs leads to an extra decision step.}\label{flow_chart_obs}
\end{figure}
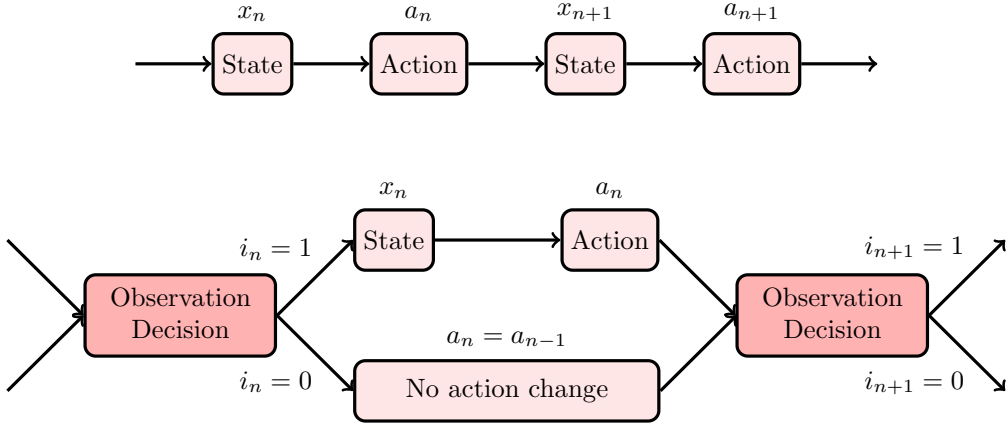

We now proceed to establish the OCM as a non-standard form of a POMDP, in order to fully make sense of \eqref{eq:formal_ocm_objective}. A policy should output an action value $a_n \in A$, as well as an inspection value $i_n \in \mathcal{I}$. The observation space is represented by $\mathcal{X} \cup \{ \varnothing \}$: either the underlying chain with values in $\mathcal{X}$ is observed, or the dummy variable $\varnothing$ nothing is observed, which represents the case of no observations. A final but subtle point is the difference in the sequential structure of the OCM compared to a POMDP. In the case of a POMDP, first a state $x_n$ is generated, follow by the observation $\bar{x}_n$, and then the action $a_n$. Thus a realisation on the canonical space of a POMDP looks like:
\begin{align*}
(x_0, \bar{x}_0, a_0, x_1, \bar{x}_1, a_1, \ldots).
\end{align*}
In the OCM, as depicted in \Cref{flow_chart_obs}, the observation occurs \textit{after} the inspection value, after which the action value follows. Thus a realisation of the system of an OCM will instead look like
\begin{align}\label{ocm_sequence}
(x_0, i_0, \bar{x}_0, a_0, x_1, i_1, \bar{x}_1, a_1, \ldots).
\end{align}
In order to obtain the sequential structure of `state - observation - action' for the OCM, we augment the sequence \eqref{ocm_sequence} with fictitious state and observation values, and treat both $a_n$ and $i_n$ as an `action'. By suitably augmenting the transition and observation kernels, the OCM takes the form of a POMDP, over the timescale of $\frac{1}{2}\N$. This augmented sequence then takes the form of 
\begin{align*}
&(x_0, \bar{x}_0, \pi_0, x_{1/2}, \bar{x}_{1/2}, \pi_{1/2}, x_1, \bar{x}_1, \pi_1, \ldots) \nonumber \\
\coloneqq&\ (x_0, \bar{x}_0, i_0, x_{0}, \bar{x}_{0}, a_0, x_1, \bar{x}_1, i_1, \ldots).
\end{align*} 
This leads us to the following definition.

\begin{defn}\label{defn: ocm}
Given an MDP $\langle \mathcal{X}, A, p, r \rangle$, the associated observation cost model (OCM) is defined as the POMDP $\langle \mathcal{X}, \mathcal{X}_{\varnothing}, \mathcal{A}, p, p_0, q, q_0, r_{\mathrm{obs}} \rangle$ (see \Cref{defn:pomdp}), on the time scale $\frac{1}{2}\N = \{0, 1/2, 1, \ldots \}$, where
\begin{itemize}
    \item $\mathcal{X}_{\varnothing} = \mathcal{X}\ \cup \{ \varnothing\}$ is the \textit{observation space}, with $\varnothing$ a dummy variable representing no observations.
    \item $\mathcal{A}$ is the \textit{disjoint} union of $A$ and $\mathcal{I}$, with time dependent admissible sets given by $\mathcal{A}(n) = \mathcal{I}$ and $\mathcal{A}(n+\frac{1}{2}) = A$;
    \item $p: \mathcal{X} \times \mathcal{A} \to \Delta_{\mathcal{X}}$ is the transition kernel, with its domain extended to $\mathcal{X} \times \mathcal{A}$ by defining
    \begin{align*}
        p(\cdot \mid x, i) = \delta_x(\cdot) ,\quad i \in \mathcal{I}.
    \end{align*}
    \item $q: \mathcal{A} \times \mathcal{X} \to \Delta_{\mathcal{X}_{\varnothing}}$ is the \textit{observation kernel},
     given by
     \begin{align*}
        q(\cdot \mid a, x) &= \delta_{\varnothing}(\cdot),  \quad a \in A, \\
         q(\cdot \mid i, x) &= i\cdot\delta_{x}(\cdot) + (1-i)\delta_{\varnothing}(\cdot), \quad i \in \mathcal{I}.     
     \end{align*}
     \item $r_{\mathrm{obs}}: \mathcal{X} \times \mathcal{A} \to \R$ is the  \textit{one-step reward function} given by
     \begin{align*}
         r_{\mathrm{obs}}(x,a) &= r(x,a),  \quad a \in A,\\
         r_{\mathrm{obs}}(x,i) &= -i\cdot c_{\mathrm{obs}}, \quad i \in \mathcal{I}. 
     \end{align*}
\end{itemize}
\end{defn}

The kernels $p$ and $q$ above are defined such that transitions of the underlying chain only occurs at the integer steps, and new observations at the half steps. Let us write the observable history sets here as
\begin{align*}
    \mathcal{H}_0 = \mathcal{X}_{\varnothing},\ \mathcal{H}_t = (\mathcal{X}_{\varnothing} \times \mathcal{A})^{2t} \times \mathcal{X}_{\varnothing},\quad t \in \frac{1}{2}\N.
\end{align*}
To be precise, we should only consider \textit{admissible} history sets, that is state-action pairs that satisfy the constraints $\mathcal{A}(n) = \mathcal{I}$ and $\mathcal{A}(n+\frac{1}{2}) = A$, but we shall take this assumption as implicit for ease of notation. A policy $\pi$ is hence a sequence of kernels $\pi_t: \mathcal{H}_t \to \mathcal{A}$.

As before, a policy $\pi$ as defined above and an initial distribution $p_0 \in \Delta_{\mathcal{X}}$ induces a unique measure $\PP^{\pi}_{p_0}$ on the canonical sample space $\Omega = (\mathcal{X} \times \mathcal{X}_{\varnothing} \times\mathcal{A})^{\infty}$. With slight abuse of notation, we also write $\pi$ for the values applied by the policy. Then, the action values $a=(a_n)_n$ and inspection values $i=(i_n)_n$ appearing in \eqref{eq:formal_ocm_objective} can be recovered by defining
\begin{align*}
i_n = \pi_{n}, \quad a_n = \pi_{n+\frac{1}{2}}.
\end{align*}

As we are assuming in the OCM that actions remain constant between new observations. We will have to consider a smaller class of admissible policies.
\begin{defn}\label{defn:policy_ocm}
An \textbf{admissible policy} $\pi = (\pi_t)_t$ for the OCM is a sequence of kernels $\pi_t: \mathcal{H}_t \to \Delta_{\mathcal{A}}$ which satisfies the following: for $n \in \N$, if $h_{n+\frac{1}{2}} = (\bar{x}_0, \pi_0, \ldots, \bar{x}_{n-\frac{1}{2}}, \pi_{n-\frac{1}{2}}, \bar{x}_n, \pi_n, \bar{x}_{n+\frac{1}{2}})$ with $\pi_n = i_{n} = 0$, then
\begin{align*}
    \pi_{n+\frac{1}{2}}\left(\ \cdot\ \middle\vert\ h_{n+\frac{1}{2}}\right) = \delta_{\pi_{n-\frac{1}{2}}} =  \delta_{a_{n-1}}.
\end{align*}
The set of admissible policies for the observation cost model is denoted $\Pi_{\mathrm{obs}}$.
\end{defn}

With the above setup, we can give a full meaning to the expression \eqref{eq:formal_ocm_objective} by writing
\begin{align*}
    J(\pi, p_0) \coloneqq \E^{\pi}_{p_0} \left[ \sum^{2N}_{n=0} r_{\mathrm{obs}} (x_{n/2}, \pi_{n/2}) \right] = \E^{\pi}_{p_0}\left[ \sum^{N}_{n=0} \left( r(x_n, a_n) -  i_n \cdot c_{\mathrm{obs}} \right)\right],\quad \pi \in \Pi_{\mathrm{obs}},\ p_0 \in \Delta_{\mathcal{X}}.
\end{align*}

\begin{rem}\label{rem:initial_kernels}
    \textbf{Regarding the initial distribution $p_0$ and observation kernel $q_0$:} for the purposes of this paper, we will assume that an observation (made at some previous time) is always available. This allows a consistent characterisation of the belief state by a finite tuple, which leads to a system of finite-dimensional QVIs in \Cref{sub:finite_horizon,sub:infinite_horizon}. Thus, we will only consider initial kernels $p_0$ that is in the form of some $n$-step transition probabilities of the kernel $p$, and the initial observation kernel $q_0$ will be taken as the Dirac measure $q_0(\cdot \mid x) = \delta_{x}(\cdot)$.
\end{rem}

\begin{rem}
The formulation with the half time-steps and the inclusion of fictitious state/ observation variables above are strictly a theoretical construct, such that the OCM can be reframed as a POMDP, and thus allowing us to directly appeal to standard results to formulate dynamic programming. It will be shown later that the half steps become redundant once dynamic programming is established, and the value function will only need to be considered over the integer steps.
\end{rem}

\subsection{Finite horizon problem}\label{sub:finite_horizon}

For the finite horizon problem, let $N \in \N$ be the time horizon. For a policy $\pi \in \Pi_{\mathrm{obs}}$, the objective function is
\begin{align*}
    \E^{\pi}_{p_0} \left[ \sum^{2N}_{n=0} r_{\mathrm{obs}} (x_{n/2}, \pi_{n/2}) \right] = \E^{\pi}_{p_0}\left[ \sum^{N}_{n=0} \left( r(x_n, a_n) -  i_n \cdot c_{\mathrm{obs}} \right)\right].
\end{align*}

With the OCM problem characterised as a POMDP, we consider the belief MDP, with the belief state given by $\mathbb{P}^{\pi}_{p_0}( x_t \mid h_t)$, where $h_t \in \mathcal{H}_t$. By the Markov property, the belief state is fully characterised by the controller's most recently observed information, in that
\begin{align*}
\mathbb{P}^{\pi}_{p_0}(x_t \mid h_t) =  p^{(\floor{t}-k)}(x_t \mid x_k, a_k),
\end{align*}
where $p^{(\floor{t}-k)}$ is the $(\floor{t}-k)$-step transition kernel of the underlying process $X$, and $k \in \N$ is the last occurrence such that $i_k = 1$. Note that $k=t$ implies an immediate observation, and the belief state reduces trivially to a Dirac measure at $x_t$. The belief state therefore has a finite dimensional parametrisation, and we can consider instead the \textbf{augmented state} $y=(y_t)_{t \in \frac{1}{2}\N}$:
    \begin{align*}
        y_t \coloneqq 
        \begin{cases}
        (k, x_k, a_k),&\ \mbox{if $ \N \ni k = \argmax_{n\leq t} \{i_{n} = 1 \} \neq t$ },\\
        (t, x_t, \varnothing),&\ \mbox{otherwise},
        \end{cases}
    \end{align*}
where $\varnothing$ also acts as a dummy variable here. The three components of $y_t$ represent the most recent observation time $k$, with the correspondingly observed state $x_k \in \mathcal{X}$, and applied action $a_k \in A$.

Given this equivalence in representation, we can consider the OCM as an MDP with the tuple $\langle \mathcal{Y}, \mathcal{A}, p_y, r_y \rangle$ on the timescale $\frac{1}{2}\N$, where
\begin{itemize}
    \item $\mathcal{Y} \coloneqq \N \times \mathcal{X} \times A_{\varnothing}$ is the \textit{augmented state space};
    \item $\mathcal{A}$ is the \textit{disjoint} union of $A$ and $\mathcal{I}$, with admissible sets $\mathcal{A}(n) = \mathcal{I}$ and $\mathcal{A}(n+\frac{1}{2}) = A$;
    \item $p_y = (p_{y,t})_{t \in \frac{1}{2}\N}$ is a (time-inhomogeneous) transition kernel on $\mathcal{Y}$ given $\mathcal{Y} \times \mathcal{A}$: for $n \in \N$, $y = (k,x,a),\ \hat{y} \in \mathcal{Y}$, $i \in \mathcal{I}$, and $a^{\prime} \in A$,
    \begin{align}
        p_{y,n}(\hat{y} \mid y , i) &= i \cdot p^{(n-k)}(\hat{x} \mid x, a) \mathbbm{1}_{\{\hat{y} = (n, \hat{x}, \varnothing) \}} + (1-i) \mathbbm{1}_{\{\hat{y}= y\}}, \label{eq:kernel_ocm_y}\\
        p_{y, n + \frac{1}{2}}(\hat{y} \mid y, a^{\prime}) &= \mathbbm{1}_{\{\hat{y} = (k, x, a^{\prime}),\ \pr{a} = a\}}, \label{eq:kernel_ocm_y_2}
    \end{align}
    and for $y = (n,x,\varnothing)$,
    \begin{align}\label{eq:kernel_ocm_y_3}
        p_{y, n + \frac{1}{2}}(\hat{y} \mid y, a^{\prime}) = \mathbbm{1}_{\{\hat{y} = (n, x, a^{\prime})\}}.
    \end{align}
    \item  $r_y = (r_{y,t})_{t \in \frac{1}{2}\N}$ is a time-dependent one-step reward function on $\mathcal{Y} \times \mathcal{A}$: for $y = (k,x,a)$, $i \in \mathcal{I}$, $a^{\prime} \in A$,
    \begin{align}
    r_{y,n}(y, i) &= - i \cdot c_{\mathrm{obs}}, \nonumber \\
    r_{y,n + \frac{1}{2}}( y, a^{\prime})& = \sum_{x^{\prime} \in \mathcal{X}} r(x^{\prime}, a^{\prime}) p^{(n-k)}(x^{\prime} \mid x, a), \label{r_y1}
    \end{align}
    and for $y = (n,x,\varnothing)$,
    \begin{align}\label{r_y2}
        r_{y,n + \frac{1}{2}}(y, a^{\prime})& = r(x, a^{\prime}).
    \end{align}
\end{itemize}

In this augmented problem, define its observable history sets as $\mathcal{H}^y_t = (\mathcal{Y} \times \mathcal{A})^{2t} \times \mathcal{Y}$, $t \in \frac{1}{2} \N$. Policies $\pi^y$ are then a sequence of kernels $\pi^y_t: \mathcal{Y} \to \Delta_{\mathcal{A}}$.  For the set of admissible policies of this augmented MDP, we will have to consider the corresponding `image' of $\Pi_{\mathrm{obs}}$. This in turn, is equivalent to imposing a state constraint on the admissible action sets: 
\begin{align*}
    \mathcal{A}(n+1/2, (k,x,a) ) = \{a \},\quad  n \in \N,\ a \in A.
\end{align*}
With the above constraints noted, we shall not distinguish between a policy $\pi \in \Pi_{\mathrm{obs}}$ and its corresponding policy $\pi^y$ in the augmented MDP, and write $\pi$ for both.

For the finite horizon problem, the objective function for the augmented state MDP is
\begin{align*}
     J(t, y, \pi) = \E^{\pi}_y\left[ \sum^{2N}_{n=t} r_{y,\frac{n}{2}}\left(y_{\frac{n}{2}}, \pi_{\frac{n}{2}}\right) \right],\quad 0 \leq t \leq 2N,\ y \in \mathcal{Y},\ \pi \in \Pi_{\mathrm{obs}},
\end{align*}
with value function
\begin{align}\label{eq:value function_mcdm_ch2}
    v(t, y) = \sup_{\pi \in \Pi_{\mathrm{obs}}} J(t,y, \pi).
\end{align}
The proposition below shows the dynamic programming equation in the form of a quasi-variational inequality, for which the value function for the OCM satisfies.
\begin{prop}
\sloppy For $n \in \N$, $y = (k,x,a) \in \mathcal{Y}$, define $v^{n,k}_{a,x} = v(n,y)$ as in \eqref{eq:value function_mcdm_ch2}. Then the value function satisfies the following quasi-variational inequality (QVI): for all ${0 \leq k < n \leq N-1}$, $x \in \mathcal{X}$, and $a \in A$,  
\begin{align} \label{qvi_finite}
    \min \bigg\{ & v^{n,k}_{a,x} - v^{n+1,k}_{a,x} - \Big(P^{n-k}_a\ r_a\Big)_x\ , \nonumber \\
    & v^{n,k}_{a,x} - \bigg( P^{n-k}_a\ \overline{v^{n+1,n} + r}\bigg)_x  +  c_{\mathrm{obs}}   \bigg\} = 0, 
\end{align}
with the terminal condition
\begin{align*}
    v^{N,k}_{a,x} &= \big(P^{N-k} r_a\big)_x,
\end{align*}
where $P^n_a$ is the $n$-step transition matrix with constant action $a$, and 
\begin{align*}
    (r_{a})_x = r(x,a), \quad \bar{r}_x = \max_{a \in A}r(x,a), \\ \left(\overline{v^{n+1,n} + r}\right)_x = \max_{a \in A} (v^{n+1,n}_{a,x} + r_{a,x}).
\end{align*}
\end{prop}

\begin{proof}
A standard application of dynamic programming gives us
\begin{align*}
    v(t,y) = \sup_{\pi \in \Pi_{\mathrm{obs}}} \left\{r_{y,t}(y, \pi_t) + \E^{\pi}\left[v\left(t + \frac{1}{2}, y_{t + \frac{1}{2}}\right) \right] \right\}.
\end{align*}
Expanding explicitly, for $n \in \N$ and $y=(k,x,a) \in \mathcal{Y}$,
\begin{align*}
    v(n, (k,x,a)) = \max \left\{ -c_{\mathrm{obs}} + \max_{a \in A}\ \E^a \left[ v\left(n + \frac{1}{2}, (n, x_n, \varnothing)\right)\right],\ v\left(n+ \frac{1}{2}, (k,x,a)\right) \right\},
\end{align*}
where $\E^a$ is the expectation taken with respect to a constant $a \in A$. Furthermore
\begin{align*}
     v\left(n+ \frac{1}{2}, (k,x,a)\right) &= \E^a \left[ r(x_n, a) \right] + v(n+1, (k,x,a)), \\
     v\left(n + \frac{1}{2}, (n, x, \varnothing)\right) &= \max_{a \in A} \left\{ r(x,a) + v(n+1, (n,x,a))\right\}
\end{align*}
Thus, by combining the inspection stage and the action stage together, and rearranging the terms accordingly, we obtain the QVI in the desired form.
\end{proof}

The optimal policy $\pi^{*}$ is then Markovian with respect to the augmented state $y$, i.e. the optimal policy depends on the most recent observation. Given a solution to the QVI \eqref{qvi_finite}, one can retrieve the optimal policy at time $n$, by first finding the region where the minimum is achieved, which determines if an inspection is optimal. If an inspection is optimal, one observes the latest state, say $x$, and the optimal action is given by $\argmax_{a \in A}(v^{n+1,n}_{a,x} + r_{a,x})$.

\begin{rem}[\textbf{On the relation between observation costs and switching costs}]\label{rem:switching}
    Due to the nature of the piecewise constant controls, the OCM resembles an optimal switching problem, where switching costs $\{g_{ij}\}_{i,j\in A}$ are present, such that $g_{ij}$ is paid each time action $i$ is changed to action $j$. There are two main differences between the two problems. Firstly, for the OCM, the optimal policy on both inspections and actions have to take into account the lack of updated observations. Secondly, in switching problems, $g_{ii}$ is assumed to be 0 for all $i \in A$, i.e., no cost is paid if no switching occurs. This contrasts to the OCM, where $c_{\mathrm{obs}} >0$ is paid at each observation, regardless of the action applied afterwards. Indeed, one may decide that no change in action is required upon observation, but nonetheless the observation cost has to be paid upfront in order to arrive at such a conclusion. We can incorporate switching costs $\{g_{ij}\}_{i,j\in A}$ into the OCM by writing the augmented reward in the form (see (\ref{r_y1}, \ref{r_y2}))
    \begin{align*}
        r_{y,n+\frac{1}{2}}(y, \pr{a}) &=  \sum_{x^{\prime} \in \mathcal{X}} r(x^{\prime}, a^{\prime}) p^{(n-k)}(x^{\prime} \mid x, a) - g_{a, \pr{a}}, \ &y =(k,x,a), \\
        r_{y,n + \frac{1}{2}}(y, a^{\prime})& = r(x, a^{\prime}) - g_{a, \pr{a}},\ &y = (n,x,\varnothing).
    \end{align*}
\end{rem}

\subsection{Infinite horizon problem}\label{sub:infinite_horizon}

For the discounted infinite horizon problem of the OCM, we will have to consider the appropriate stationary formulations. This can easily be obtained by further considering $(n, y_n)$ as an augmented state. Now recall that the transition kernel of $y$ in \Cref{eq:kernel_ocm_y,eq:kernel_ocm_y_2,eq:kernel_ocm_y_3}
depends on $n$ and $k$ strictly through the difference $n-k$. Hence, after relabelling, it is sufficient to consider $ y = (n,x,a)$ as the augmented state, where here $n$ now represents the time elapsed from the previous observation, rather than the standard linear passage of time. The objective function of this equivalent MDP is
\begin{align*}
     J(y, \pi) = \E^{\pi}\left[ \sum^{\infty}_{n=0} \gamma^{\floor{\frac{n}{2}}} r_y\left(y_{\frac{n}{2}}, \pi_{\frac{n}{2}}\right) \right],\quad y \in \mathcal{Y},\ \pi \in \Pi_{\mathrm{obs}},
\end{align*}
where $\gamma \in (0,1)$ is the discount factor, the value function is
\begin{align}\label{eq:value function_inf_mcdm_ch2}
    v(y) = \sup_{\pi \in \Pi_{\mathrm{obs}}} J(y, \pi).
\end{align}
This gives us the following QVI for the value function, which we shall state here without proof.

\begin{prop}\label{prop:inf_dpp}
For $y = (n,x,a) \in \mathcal{Y}$, define $v^{n}_{a,x} = v(y)$. Then the value function \eqref{eq:value function_inf_mcdm_ch2} satisfies the following quasi-variational inequality (QVI): for all $n \geq 1$, $x \in \mathcal{X}$, and $a \in A$,  
\begin{align} \label{qvi_infinite}
    \min \bigg\{ & v^{n}_{a,x} - \gamma v^{n+1}_{a,x} - \Big(P^{n}_a\ r_a\Big)_x\ , \nonumber \\
    & v^{n}_{a,x} - \bigg( P^{n}_a\ \overline{\gamma v^{1} + r}\bigg)_x  +  c_{\mathrm{obs}}   \bigg\} = 0, 
\end{align}
where $P^n_a$ is the $n$-step transition matrix with constant action $a$, and
\begin{align*}
    (r_{a})_x = r(x,a), \quad \bar{r}_x = \max_{a \in A}r(x,a) \\
    \left(\overline{\gamma v^{1} + r}\right)_x = \max_{a \in A} (\gamma v^1_{a,x} + r_{a,x}).
\end{align*}
\end{prop}

Note that the QVI \eqref{qvi_infinite} is defined on the infinite domain $\N_{\geq 1} \times \mathcal{X} \times A$. In practice, we will have to truncate this domain for the time variable. A natural boundary condition is to enforce an inspection of the underlying chain after some large time $N$ has elapsed. This is equivalent to further restricting the admissible policies in $\Pi_{\mathrm{obs}}$ to those such that $i_N = 1$.

\subsection{Approximating the infinite horizon QVI by truncation}

Note that the QVI \eqref{qvi_infinite} is defined on the infinite domain $\N_{\geq 1} \times \mathcal{X} \times A$. In practice, we will have to truncate this domain for the time variable. A natural boundary condition is to enforce an inspection of the underlying chain after some large time $N$ has elapsed.

We can justify the truncation in the following way. For the original untruncated problem, we can write the reward functional and value function as
\begin{align*}
    J(\pi, p)  &= \E^{\pi}_{p}\left[ \sum^{\infty}_{n=0} \gamma^n \left( r(x_n, a_n) -  i_n \cdot c_{\mathrm{obs}} \right)\right],\quad \pi \in \Pi_{\mathrm{obs}},\ p \in \Delta_{\mathcal{X}},\\
    v(p)&= \sup_{\pi\in\Pi_{\mathrm{obs}}} J(\pi,p).
\end{align*}

For the truncated problem, enforcing a maximum time for inspection is equivalent to further restricting the set of admissible policies within $\Pi_{\mathrm{obs}}$. To this end, define $\Pi^N_{\mathrm{obs}} \subset \Pi_{\mathrm{obs}}$ such that any $\pi^N = (a^N, i^N) \in \Pi^N_{\mathrm{obs}}$ satisfies
\begin{align*}
    \PP( \min\{n > k: i^N_n = 1 \} \leq k + N \mid i^N_k = 1) = 1\quad \mbox{for all $k\geq 0$.}
\end{align*}

Then we can define the truncated value function
\begin{align*}
    v^N(p) \coloneqq \sup_{\pi \in \Pi^N_{\mathrm{obs}}} J(\pi, p).
\end{align*}

\begin{prop}
    The truncated value functions $v^N$ are monotonically increasing in $N$, and furthermore $v^N \uparrow v$ pointwise. 
\end{prop}

\begin{proof}
    The first part is clear as $\Pi^M_{\mathrm{obs}} \subset \Pi^N_{\mathrm{obs}} \subset \Pi_{\mathrm{obs}}$ for any $M<N$, so that $v^N$ is monotonically increasing and bounded above by $v$. Therefore $v^N$ converges to a limit $v^{*}$ and 
    \begin{align*}
        v^{*}(p) = \lim_{N\to\infty}v^N(p) \leq v(p).
    \end{align*}
    To show the reverse inequality, let $\{\pi_n\}\subset \Pi_{\mathrm{obs}}$ be a maximising sequence for $v$. As both $\mathcal{X}$ and $A$ are finite, for each $n >0$, we can choose $N_n \in \N$ such that
    \begin{align*}
        \left\lvert \sum^{\infty}_{n=N_n} \gamma^n (r(x,a) - c_{\mathrm{obs}})\right\rvert \leq 2^{-n-1}\ \mbox{ for all $x\in\mathcal{X}$ and $a \in A$.}
    \end{align*}
    Then, choose $\pi^N_n\in\Pi^N_{\mathrm{obs}}$ such that $\pi^N_n(k) = \pi_n(k)$ for all $k \leq N_n$, so that
    \begin{align*}
        J(\pi_n, p) &\leq J(\pi^N_n, p) + 2^{-n} \\
        &\leq v^{N_n}(y) + 2^{-n}.
    \end{align*}
    But since $\{\pi_n\}$ is a maximising sequence for $v$, and $N_n \to \infty$ by construction, letting $n\to\infty$ on both sides leads to
    \begin{align*}
        v(y) \leq \lim_{n\to\infty} v^{N_n}(y) = v^{*}(y),
    \end{align*}
    as required.
\end{proof}

\subsection{Parametrised time-inhomogeneous actions}\label{sub:inhomogeneous}

\sloppy 
Here, we show how the assumption of constant controls between observations can be relaxed to allow for
parametrised time-dependent actions.
Optimising over open-loop actions in general leads to an unbounded number of parametrising variables for the belief state $y$. For tractability, we consider instead a set of parametrised time-inhomogeneous open-loop actions. This can be incorporated into the previous framework as follows. Let $\Theta$ be a finite parameter set. Let $\{f_{\theta}\}_{\theta \in \Theta}$ be a parametrised set of open-loop actions, with $f_{\theta}: \N \to A$. Then we write the OCM POMDP as $\langle \mathcal{Y}, \mathcal{A}, p_y, r_y \rangle$ on the timescale $\frac{1}{2}\N$, but now with $\mathcal{A}$ being the union of $\Theta$ and $\mathcal{I}$, with admissible sets $\mathcal{A}(n) = \mathcal{I}$ and ${\mathcal{A}(n+\frac{1}{2}) = \Theta}$. Then the augmented transition kernel $p_y = (p_{y,t})_{t \in \frac{1}{2}\N}$ can be written, for $n \in \N$, $y = (k,x,\theta)$, $\hat{y} \in \mathcal{Y}$, $i \in \mathcal{I} = \{0, 1\}$, and $\theta^{\prime} \in \Theta$, as
    \begin{align*}
        p_{y,n}(\hat{y} \mid y , i) &= i \cdot p^{(n-k)}(\hat{x} \mid x, f_{\theta}) \mathbbm{1}_{\{\hat{y} = (n, \hat{x}, \varnothing) \}} + (1-i) \mathbbm{1}_{\{\hat{y}= y\}},\\
        p_{y, n + \frac{1}{2}}(\hat{y} \mid y, \theta^{\prime}) &= \mathbbm{1}_{\{\hat{y} = (k, x, \theta^{\prime}),\ \pr{a} = a\}},
    \end{align*}
    and for $y = (n,x,\varnothing)$,
    \begin{align*}
        p_{y, n + \frac{1}{2}}(\hat{y} \mid y, \theta^{\prime}) = \mathbbm{1}_{\{\hat{y} = (n, x, \theta^{\prime})\}},
    \end{align*}
where the term $p^{(n-k)}(\hat{x} \mid x, f_{\theta})$ represents the $(n-k)$-step transition kernel starting from $x \in \mathcal{X}$, applying the successive actions $f_{\theta}(1), \ldots, f_{\theta}(n-k)$. A similar replacement can be done with the augmented reward $r_y$. Then the QVI can be rewritten, for the infinite horizon case for example, as
\begin{align*}
    \min \bigg\{ & v^{n}_{\theta,x} - \gamma v^{n+1}_{\theta,x} - \Big(Q^{n}_{\theta}\ r_{f_{\theta}(n)}\Big)_x\ , \nonumber \\
    & v^{n}_{\theta,x} - \bigg( Q^{n}_{\theta}\ \overline{\gamma v^{1} + r}\bigg)_x  +  c_{\mathrm{obs}}   \bigg\} = 0, 
\end{align*}
with
$Q^n_\theta = \prod^n_{k=1} P_{f_{\theta}(k)}$.

\color{black}

\subsection{Toy problem}\label{toy_example}

To illustrate the framework, we present a model problem involving a two-state Markov chain and give an explicit solution. We assume the following setup:
\begin{itemize}
    \item the state space $\mathcal{X}=\{0,1\}$;
    \item the action space $A=\{0,1\}$;
    \item the reward function $r(x, a)= a \cdot x + (1-a)(a-x)$;
    \item the transition matrix
    \begin{equation*}
    P_{a}=\begin{blockarray}{ccc}
    0 & 1 \\
    \begin{block}{[cc]c}
    a + p(1-a) & (1-p)(1-a) & 0 \\
    (1-p) a & p \cdot a + (1-a) & 1 \\
    \end{block}
    \end{blockarray}
\end{equation*}
where $p\in(0,1)$.
\end{itemize}

This can be seen as a model for a maintenance problem, to find an optimal interval for inspecting equipment to avoid wear and tear over time. The reward function $r$ gives a reward of 1 when the state and action values are the same, and zero otherwise. If no changes are made to the action, the chain eventually arrives at the absorbing state which does not incur any reward. \Cref{toymarkov} illustrates the chain for the case $p=0.9$. 

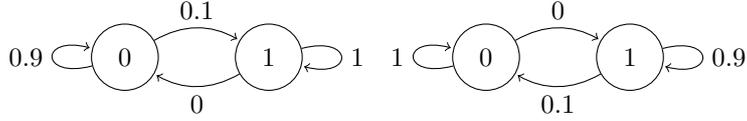
\begin{figure}[ht]
\centering
\begin{tikzpicture}
    \node[state] (0) {0};
    \node[state, right=of 0] (1) {1};
    \draw[every loop]
            (0) edge[bend left, auto=left] node {0.1} (1)
            (1) edge[bend left, auto=left] node {0} (0)
            (0) edge[loop left] node {0.9} (0)
            (1) edge[loop right] node {1} (1);
\end{tikzpicture}
\begin{tikzpicture}
    \node[state] (0) {0};
    \node[state, right=of 0] (1) {1};
    \draw[every loop]
            (0) edge[bend left, auto=left] node {0} (1)
            (1) edge[bend left, auto=left] node {0.1} (0)
            (0) edge[loop left] node {1} (0)
            (1) edge[loop right] node {0.9} (1);
\end{tikzpicture}
\caption{Illustration of the two-state Markov chain. Left: $a=0$; Right: $a=1$.}
\label{toymarkov}
\end{figure}

Consider the infinite horizon problem. The QVI for $x=0$ and $a=0$ is
\begin{align}\label{toy_qvi}
    \min \bigg\{ & v^{n}_{0,0} - \gamma v^{n+1}_{0,0} - \Big(P^{n}_0\ r_0\Big)_x\ , \nonumber \\
    & v^{n}_{0,0} - \left( P^{n}_0 \overline{\gamma v^{1} +r}\right)_x  +  c_{\mathrm{obs}}   \bigg\} = 0.
\end{align}
Due to the symmetry of the problem, we have $v^n_{0,0} = v^n_{1,1}$. Moreover, it is clear that given the knowledge of $x_n$, the optimal action is to set ${a_n = x_n}$. Hence we can write \eqref{toy_qvi} as
\begin{align*}
\min \{ v^n_{0,0} - \gamma v^{n+1}_{0,0} - p^n,\  v^n_{0,0} - \gamma v^1_{0,0} -1 + c_{\mathrm{obs}} \} = 0
\end{align*}
or simply by writing $v(n) = v^n_{0,0}$:
\begin{equation*}
    v(n) = \max \{ p^{n} + \gamma  v(n+1),\  1 -c_{\mathrm{obs}} + \gamma  v(1) \}.
\end{equation*}
Let $T$ be the first optimal inspection time (where by convention $T=\infty$ if it is optimal to never inspect). The value function is given recursively by
\begin{equation}\label{inf_sol}
    v(n) = \begin{cases}
        1 -c_{\mathrm{obs}}+\gamma v(1),& \quad \mbox{if }n\geq T; \\
        p^n + \gamma v(n+1),& \quad \mbox{otherwise}.
    \end{cases}
\end{equation}

Solving the above for $v(1)$, we obtain the explicit solution
\begin{equation}\label{inf_sol_1}
    v(1) =\max{\left\{\sup_{m\geq 2}{\left(\frac{p \sum^{m-2}_{k=0} (\gamma p)^k + \gamma^{m-1}(1-c_{\mathrm{obs}})}{1-\gamma^m}\right)},\frac{1-c_{\mathrm{obs}}}{1-\gamma}\right\}},
\end{equation}
from which $v(n)$ for $n\geq1$ can be calculated from \eqref{inf_sol}. The first term in \eqref{inf_sol_1} is a geometric series, where $p \sum^{m-2}_{k=0} (\gamma p)^k + \gamma^{m-1}(1-c_{\mathrm{obs}})$ is the expected returns across the optimal inspection interval, and $\gamma^{m}$ is the discount factor over the whole interval. We can interpret \eqref{inf_sol_1} as searching for the optimal inspection interval to maximise the sum of the rewards, minus the observation cost. 

\section{Comparison principle and penalisation}\label{section_comparison}

In this section, we consider the well-posedness of solutions to the QVI for the OCM. We establish uniqueness by proving a comparison principle, and existence by constructing arbitrarily close approximations via truncation and penalisation. To this end, we shall recast the QVIs in \Cref{section_mc_model} in a more abstract form.

\begin{prob}\label{prob:original_qvi}
\sloppy Find a function $u: \N \times \mathcal{X} \to \R^{A}$, with $u_a(\cdot,\cdot)$ as its components, satisfying
\begin{align}\label{abstract_qvi_infinite}
    \min \left\{ F_a(n,x,u(n,x),u_a),\ u_a(n,x) - \mathcal{M} u(n,x) \right\}=0, \quad  n\in\N,\ x \in \mathcal{X},\ a \in A,
\end{align}
where $F_a: \N \times \mathcal{X} \times \R^{A} \times \R^{\N\times \mathcal{X}} \to \R$ satisfies \Cref{ass_monotonicity}, and ${\mathcal{M}u: \N \times \mathcal{X} \to \R}$ is defined for all
$(n,x) \in \N \times \mathcal{X}$ by 
    \begin{align}\label{obstacle operator}
        \mathcal{M}u(n,x) = \left(Q_n \overline{u^{1} - c}\right)_x ,\quad \left(\overline{u^{1} - c}\right)_x = \max_{a \in A} \left(u^1_{a,x} - c_{a,x}\right),
    \end{align}
for a given vector $c \in \R^{\mathcal{X} \times A}$, and $\{Q_n\} \subset \mathbb{R}^{\mathcal{X}\times \mathcal{X}}$ is a sequence of substochastic matrices, i.e., with non negative row sums bounded by 1. 
\end{prob}

\begin{ass}\label{ass_monotonicity}
The functions $\{F_a\}_{a \in A}$ in Problem \ref{prob:original_qvi} satisfy:
    \begin{enumerate}
        \item There exists $\beta > 0$ such that for any $n \in \N$, $x \in \mathcal{X}$, bounded functions $u, v: \N \times \mathcal{X} \to \R^A$ such that $u \leq v$, and $r, s \in \R^A$ such that $\theta \coloneqq r_a - s_a = \max_{j \in A} (r_j - s_j) \geq 0$, then 
        \begin{align*}
            F_a(n, x, r, u_a + \theta) - F_a(n, x, s, v_a) \geq \beta \theta.
        \end{align*}
        \item For any $n \in N$, $a\in A$, $r \in \R^A$ and continuous bounded function $\phi: \N \times \mathcal{X} \to \R$, the functions $x \mapsto F_a(n, x, r, \phi)$ are continuous and bounded. Moreover, the functions $r \mapsto F_a(n,x, r, \phi)$ are uniformly continuous for bounded $r$, uniformly with respect to $(n, x)\in \N \times \mathcal{X}$.
    \end{enumerate}
\end{ass}

\sloppy \Cref{ass_monotonicity} are monotonicity and regularity assumptions typically used in viscosity solution theory \cite{monotone_system_convergence}, which applies to standard approximation schemes, such as finite differences. The monotonicity condition arises naturally, for example, from the discretisation of QVIs in continuous time involving switching controls. The operator $\mathcal{M}$, which shall be referred to as the \textit{inspection operator}, is a non-local term in the QVI that couples the solution $u$ across the different action values. In general, the vector $c$ in \eqref{obstacle operator} can also depend on $n$, and the proofs in this section extends naturally to that case.

We now prove a comparison principle for the QVI \eqref{abstract_qvi_infinite}, extending a similar result in \cite{Reisinger_Zhang_penalty} for finite QVIs and with the identity matrix in place of $Q_n$. This establishes the uniqueness of solutions to \Cref{prob:original_qvi}.

\begin{prop}\label{comparision_thm_infinite}
Given \Cref{ass_monotonicity}, suppose $u: \N \times \mathcal{X} \to \R^A$ (resp. $v$) is bounded and satisfies for $c>0$,
    \begin{align*}
         \min \left\{ F_a(n,x,u(n,x),u_a),\ u_a(n,x) - \mathcal{M} u(n,x) \right\}\leq 0 \quad \mbox{(resp. $\geq 0$)}.
    \end{align*}
    Then we have $u \leq v$.
\end{prop}

\begin{proof}
    We write $u^n_{a,x}$ for $u_a(n,x)$, and let $M = \sup_{n,x,a}(u^n_{a,x}-v^n_{a,x})$. Let $(n_k, x_k, a_k)$ be a maximising sequence, with $M_k = u^{n_k}_{a_k, x_k} - v^{n_k}_{a_k, x_k}$. Suppose that along the whole sequence, $u_{a_k} - \mathcal{M}u \leq 0$. As $v$ is a supersolution, we also have that $v_{a_k} - \mathcal{M}v \geq 0$. Therefore,
    \begin{align*}
        u^{n_k}_{a_k, x_k} - v^{n_k}_{a_k, x_k} &\leq \left(Q_{n_k} \overline{u^{1} - c}\right)_{x_k} - \left(Q_{n_k} \overline{v^{1} - c}\right)_{x_k}&  \\
    & \leq \left(Q_{n_k} \overline{u^{1} - v^1}\right)_{x_k} &\\
    & \leq \left(\overline{u^{1} - v^1}\right)_{x^{*}}\quad\quad &\mbox{(for some $x^{*} \in \mathcal{X}$)}.
    \end{align*}
    As the RHS is independent of $k$, taking the supremum on the LHS leads to
    \begin{align*}
        M \leq  \left(\overline{u^{1} - v^1}\right)_{x^{*}} \leq \max_{a\in A}\ (u^1_{a, x^{*}} - v^1_{a, x^{*}}).
    \end{align*}
    As $A$ is finite, the supremum for $M$ is achieved at some $(1, b_1, y_1)$, $b_1 \in A$, $y_1 \in \mathcal{X}$. Applying the same argument as above, we have
    \begin{align*}
         u^{1}_{b_1, y_1} - v^{1}_{b_1, y_1} \leq -c \leq u^1_{b_2, y_2} - v^1_{b_2,y_2},
    \end{align*}
    for some $(b_2, y_2) \neq (b_1, y_1)$, since by assumption $c \geq 0$. Repeating the same argument leads to successive different indices $(b_i, y_i)$ such that
    \begin{align}\label{eq:contradiction_comparison}
        u^1_{b_1, y_1} - u^1_{b_n, y_n} \leq \sum^{n-1}_{j=1} \left(u^1_{b_j, y_j} - u^1_{b_{j+1}, y_{j+1}}\right) \leq -(n-1) c < 0.
    \end{align}
    But since both $\mathcal{X}$ and $A$ are finite, by the pigeonhole principle, we must eventually have some $n$ such that $(b_1, y_1)= (b_n, y_n)$, which contradicts $\eqref{eq:contradiction_comparison}$. Therefore, we must have some $(n_k, x_k, a_k)$ such that
    \begin{align}\label{monotonicity_contradict}
        F_{a_k}(n_k, x_k, u(n_k,x_k), u_{a_k}) - F_{a_k}(n_k, x_k, v(n_k,x_k), v_{a_k}) \leq 0.
    \end{align}
    Now by construction, $u_{a_k} \leq v_{a_k} +M$. Hence, by taking $r=u(n_k,x_k)$, $s=v(n_k, x_k) +M_k$ and $\theta \coloneqq \max_{a \in A}(r_a - s_a) = 0$ in the monotonicity condition in \Cref{ass_monotonicity}, we have
    \begin{align*}
        F_{a_k}(n_k, x_k, v(n_k,x_k) + M_k, v_{a_k} +M) \leq F_{a_k}(n_k, x_k, u(n_k,x_k), u_{a_k}).
    \end{align*}
    Next, as $r \mapsto F_a(n,x, r, \phi)$ is uniformly continuous, it admits a modulus of continuity, which we denote by $\omega(\vert r-s\vert)$, so that we have
    \begin{align*}
        F_{a_k}(n_k, x_k, v(n_k,x_k) + M, v_{a_k} +M) - \omega(M-M_k) \leq F_{a_k}(n_k, x_k, v(n_k,x_k) + M_k, v_{a_k} +M).
    \end{align*}
    Combining the above with \eqref{monotonicity_contradict}, we arrive at
    \begin{align*}
         \beta M - \omega(M-M_k) \leq F_{a_k}(n_k, x_k, u(n_k,x_k), u_{a_k}) - F_{a_k}(n_k, x_k, v(n_k,x_k), v_{a_k}) \leq 0.
    \end{align*}
    Then, as $k \to \infty$, we arrive at a contradiction, since the modulus of continuity $\omega \to 0$ and $\beta M>0$ by assumption.
\end{proof}

\begin{cor}
The solutions to the QVI for the OCM infinite horizon problem \eqref{qvi_infinite} are unique.
\end{cor}
\begin{proof}
Clearly \eqref{qvi_infinite} can be cast in the form of \eqref{abstract_qvi_infinite}, given by
    \begin{align*}
        F_a(n,x,u,\phi) & = u - \gamma \phi(n+1,x) - \left( P^n_a r_a \right)_x, \\
        Q_n & = \gamma P^n_a,\\
        c_{a,x} &= c_{\mathrm{obs}} - \frac{1}{\gamma} r_{a,x}.
    \end{align*}
    Moreover, $F_a$ satisfies \Cref{ass_monotonicity} with $\beta = 1 - \gamma$. Since a solution to \eqref{abstract_qvi_infinite} is both a subsolution and supersolution, it follows from \Cref{comparision_thm_infinite} that any solution must be unique.
\end{proof}

As alluded to in \Cref{section_mc_model}, in order to compute the solutions to the QVI \eqref{abstract_qvi_infinite}, we have to resort to computing a finite QVI, truncated at some large time. Here we shall also show the comparison principle for the truncated QVI. Let $N$ be the truncated time, $\lvert \mathcal{X} \rvert = L$, $\lvert A \rvert =d$. Then a function $u: \{1,\ldots, N\} \times \mathcal{X} \to \R^A$ can be considered as an element of $\R^{N\times L \times d}$. Therefore, given $F_a$ in \Cref{prob:original_qvi}, consider its restriction to $\{1,\ldots, N\}\times\mathcal{X}\times\R\times \R^{N\times L \times d}$, and write
\begin{align*}
    F_a(n,x,u(n,x), u_a) \eqqcolon \tilde{F}_a(u)^n_x,
\end{align*}
i.e., we consider $\{F_a\}_{a \in A}$ as functions from $\R^{N \times L \times d}$ to $\R^{N \times L}$. The monotonicity condition can then be reduced to the form below.
\begin{lemma}
     \sloppy $\tilde{F}_a: \R^{N \times L \times d} \to \R^{N \times L}$ is a continuous function that satisfies the following property: there exists a constant $\beta > 0$ such that for any $u,v\in \R^{N \times L \times d}$ with ${u^{\bar{n}}_{\bar{a},\bar{l}} - v^{\bar{n}}_{\bar{a},\bar{l}} = \max_{n,a,l} (u^n_{a,l} - v^n_{a,l}) \geq 0}$, we have
\begin{equation}\label{ass_finite_monotonicity}
    \tilde{F}_{\bar{a}}(u)^{\bar{n}}_{\bar{l}} - \tilde{F}_{\bar{a}}(v)^{\bar{n}}_{\bar{l}} \geq \beta (u^{\bar{n}}_{\bar{a},\bar{l}} -v^{\bar{n}}_{\bar{a},\bar{l}}).
\end{equation}
\end{lemma}

\begin{proof}
    For the first part, let $u, v \in \R^{N \times L \times d}$ with $
    \theta \coloneqq {u^{\bar{n}}_{\bar{a},\bar{l}} - v^{\bar{n}}_{\bar{a},\bar{l}} = \max_{n,a,l} (u^n_{a,l} - v^n_{a,l}) \geq 0}$. Then by definition of $\theta$, $\tilde{u} \coloneqq u - \theta\leq v $. Applying \Cref{ass_monotonicity} with $r = u(\bar{n}, \bar{l})$, $s = v(\bar{n}, \bar{l})$ and functions $\tilde{u}, v$ lead to
    \begin{align}\label{mono_forward}
        F_{\bar{a}}(\bar{n},\bar{l}, u(\bar{n},\bar{l}), \tilde{u}_{\bar{a}} + \theta) - F_{\bar{a}}(\bar{n},\bar{l}, v(\bar{n},\bar{l}), v_{\bar{a}}) \geq \beta \theta.
    \end{align}
    Noting that $\tilde{u}_{\bar{a}} + \theta = u_{\bar{a}}$, it then follows from \eqref{mono_forward} that $\tilde{F}_a(u)^n_x = F_a(n,x,u(n,x), u_a)$ satisfies \eqref{ass_finite_monotonicity} as required. 
\end{proof}

Thus, the finite truncated QVI can be summarised by the problem below.
\color{black}
\begin{prob}\label{prob:QVI}
Find $ u = (u_1, \ldots, u_d) \in \R^{N \times L \times d}$ such that 
\begin{align}\label{discrete_qvi}
    \min \left\{ \tilde{F}_a(u), u_a - \mathcal{M} u \right\}=0, \quad a \in \{1,\ldots , d\} \eqqcolon A,
\end{align}
with time boundary condition
\begin{align*}
    u_a(N, \cdot) = \mathcal{M}u(N, \cdot),
\end{align*}
where $\tilde{F}_a: \R^{N \times L \times d} \to \R^{N \times L}$ satisfies \eqref{ass_finite_monotonicity}, and $\mathcal{M}: \R^{N \times L \times d} \to \R^{N \times L}$ is defined by 
        \begin{align*}
            (\mathcal{M} u)^n_l = \left(Q_n \overline{u^{1} - c}\right)_l ,\quad \left(\overline{u^{1} - c}\right)_l = \max_{a \in A} \left(u^1_{a,l} - c_{a,l}\right),
        \end{align*}
        for a given vector $c \in \R^{L \times d}$ and $\{Q_n\} \subset \mathbb{R}^{L\times L}$ is a sequence of substochastic matrices.
\end{prob}

\begin{rem}
    For the case with parameter uncertainty, if the measures $\rho_n(d \theta)$ can be parametrised by a finite number of parameters $w$, this can also be considered as part of the spatial domain. In this case $L = \lvert \mathcal{X} \rvert \cdot \lvert w \rvert$. 
\end{rem}

When the domain is finite as above, and in the case where the $Q_n$'s are the identity matrix, then \eqref{discrete_qvi} reduces to a QVI with interconnected obstacles, see \cite{Reisinger_Zhang_penalty} for a more detailed analysis for such classes of QVIs. Naturally, the truncated QVI also satisfies a comparison principle:

\begin{prop} \label{comparison_qvi}
Suppose $u=(u_a)_{a \in A}$ (resp. $v=(v_a)_{a \in A}$) satisfies for $c>0$:
    \begin{equation*}
        \min \left\{ \tilde{F}_a(u),\ u_a - \mathcal{M}u \right\} \leq 0 \quad (\mathrm{resp.} \geq 0),\quad a \in A;
    \end{equation*}
then $u \leq v$.
\end{prop}

\begin{proof} Let $M \coloneqq u^{\bar{n}}_{\bar{a},\bar{l}} - v^{\bar{n}}_{\bar{a},\bar{l}} = \max_{n,a,l} (u^n_{a,l} - v^n_{a,l})$. The proof follows directly from the arguments in \Cref{comparision_thm_infinite}. Assuming that the subsolution $u$ satisfies $u_{\bar{a}} \leq \mathcal{M} u$ leads to
\begin{align*}
     u^{\bar{n}}_{\bar{a},\bar{l}} - v^{\bar{n}}_{\bar{a},\bar{l}} \leq -c \leq u^1_{a^{*}, l^{*}} - v^1_{a^{*}, l^{*}}\quad \mbox{for some $a^{*}$ and $l^{*}$,}
\end{align*}
so that the maximum is also achieved at $(1, a^{*}, l^{*})$. Then the same contradiction argument leads to $\tilde{F}_{\bar{a}}(u)^{\bar{n}}_{\bar{l}} \leq 0$, but then since $v$ is a supersolution and $M > 0$, we have by the monotonicity property
\begin{equation*}
    \beta (u^{\bar{n}}_{\bar{a},\bar{l}} - v^{\bar{n}}_{\bar{a},\bar{l}}) \leq \tilde{F}_{\bar{a}}(u)^{\bar{n}}_{\bar{l}}- \tilde{F}_{\bar{a}}(v)^{\bar{n}}_{\bar{l}} \leq 0
\end{equation*}
which is again a contradiction. Hence we must have $M \leq 0$ as required.
\end{proof}

We now present a penalty approximation to the QVI \eqref{discrete_qvi}. The motivations behind this are threefold. Firstly, the penalty method holds a crucial advantage over a more traditional policy iteration approach, as using the latter can lead to numerical instabilities or indeed ill-defined iterates, due to the lack of a guarantee of the invertibility of the matrix arising from the intervention operator $\mathcal{M}$. A simple example of such an instance is given in \cite[Section 6]{reisinger2020error}, with a more in depth explanation given in \cite{weakly_chained}. 
Secondly, for QVIs that arise from the discretisation of diffusions on continuous state spaces, the number of policy iterations is typically large for fine meshes, as exposed in \cite{reisinger2012use}, while semismooth Newton iterations for the penalised system are robust under mesh refinement. Lastly, penalisation provides naturally a constructive proof of existence to the solutions of the QVI \eqref{discrete_qvi}. Consider the following penalised problem.

\begin{prob}\label{penalised_problem}
Let $\rho \geq 0$ be the penalty parameter. Find $u^{\rho} = (u^{\rho}_a)_{a \in A} \in \R^{N \times L \times d}$ such that 
\begin{equation}\label{penalised_eq}
    G^{\rho}_a(v) \coloneqq \tilde{F}_a(u^{\rho}) - \rho\ \pi \left(\mathcal{M}u^{\rho} - u^{\rho}_a \right) = 0,\ a \in A,
\end{equation}
where the penalisation function $\pi: \R \to \R$ is continuous, non-decreasing with $\pi\vert_{(-\infty,0]} = 0$ and $\pi\vert_{(0,\infty)}>0$, and is applied elementwise.
\end{prob}

Thus in the penalised problem, a penalty $\rho$ is applied whenever the condition ${u^{\rho}_a - \mathcal{M}u^{\rho} \geq 0}$ is violated. As $\rho \uparrow \infty$, the penalised solution should then converge to the solution of the discrete QVI \eqref{discrete_qvi}. We first show below that for each fixed $\rho$, \eqref{penalised_eq} satisfies a comparison principle. This implies uniqueness for \Cref{penalised_problem}. The argument follows similarly to the approach in \cite{Reisinger_Zhang_penalty} and \Cref{comparison_qvi}.

\begin{prop}\label{comparison_penalty}
For any penalty parameter $\rho\geq 0$, if $u^{\rho} = (u^{\rho}_a)_{a \in A}$ (resp., $v^{\rho} = (v^{\rho}_a)_{a \in A}$) satisfies
\begin{equation*}
    \tilde{F}_a( u^{\rho}) - \rho\ \pi \left(\mathcal{M}u^{\rho} - u^{\rho}_a \right)  \leq 0 \quad (\mathrm{resp.,} \geq 0),
\end{equation*}
then $u^{\rho}\leq v^{\rho}$.
\end{prop}

\begin{proof}
Let $M \coloneqq u^{\rho,\bar{n}}_{\bar{a},\bar{l}} - v^{\rho,\bar{n}}_{\bar{a},\bar{l}} = \max_{n,a,l} (u^{\rho,n}_{a,l} - v^{\rho,n}_{a,l})$. Suppose for a contradiction that $M > 0$. From the previous proposition, we have that 
\begin{align*}
\left(Q_{\bar{n}} \overline{u^{\rho,1} - c}\right)_{\bar{l}} - \left(Q_{\bar{n}} \overline{v^{\rho,1} - c}\right)_{\bar{l}}
    \leq u^{\rho,\bar{n}}_{\bar{a},\bar{l}} - v^{\rho,\bar{n}}_{\bar{a},\bar{l}}.
\end{align*}
As $\pi$ is non-decreasing,
\begin{equation}\label{eq:pen_comp_pi}
     \pi \left( \left(Q_{\bar{n}} \overline{u^{\rho,1} - c}\right)_{\bar{l}}  -  u^{\rho,\bar{n}}_{\bar{a},\bar{l}} \right) \leq  \pi \left( \left(Q_{\bar{n}} \overline{v^{\rho,1} - c}\right)_{\bar{l}}- v^{\rho,\bar{n}}_{\bar{a},\bar{l}}\right).
\end{equation}
As $u^{\rho}$ and $v^{\rho}$ are respectively sub and super solutions of \eqref{penalised_eq}, we have
\begin{align*}
    & \tilde{F}_{\bar{a}}(u^{\rho})^{\bar{n}}_{\bar{l}} -  \rho\ \pi \left(\left(\mathcal{M}u^{\rho}\right)^{\bar{n}}_{\bar{l}} -  u^{\rho,\bar{n}}_{\bar{a},\bar{l}} \right) - \left( \tilde{F}_{\bar{a}}(v^{\rho})^{\bar{n}}_{\bar{l}} -  \rho\ \pi \left( \left(\mathcal{M}v^{\rho}\right)^{\bar{n}}_{\bar{l}} - v^{\rho,\bar{n}}_{\bar{a},\bar{l}} \right) \right) \\
    \leq\ & \tilde{F}_{\bar{a}}(u^{\rho})^{\bar{n}}_{\bar{l}} - \tilde{F}_{\bar{a}}(v^{\rho})^{\bar{n}}_{\bar{l}} - \rho \left(  \pi \left( \left(Q_{\bar{n}} \overline{u^{\rho,1} - c}\right)_{\bar{l}}  -  u^{\rho,\bar{n}}_{\bar{a},\bar{l}} \right)  -   \pi \left( \left(Q_{\bar{n}} \overline{v^{\rho,1} - c}\right)_{\bar{l}}- v^{\rho,\bar{n}}_{\bar{a},\bar{l}}\right) \right)\\
    \leq\ & 0
\end{align*}
Hence by \eqref{eq:pen_comp_pi}, $\tilde{F}_{\bar{a}}(u^{\rho})^{\bar{n}}_{\bar{l}} -  \tilde{F}_{\bar{a}}(v^{\rho})^{\bar{n}}_{\bar{l}} \leq 0$. The monotonicity assumption of $F$ then leads to a contradiction, so that $M \leq 0 $ as required.
\end{proof}

Next, we show that the penalised solutions are uniformly bounded above, which is crucial for the convergence towards the unpenalised problem.
\begin{lemma}\label{lem_bound}
    For $\rho \geq 0$ and $c \geq 0$, let $u^{\rho}$ be the solution to \eqref{penalised_eq}. Then $\lVert u^{\rho} \rVert \leq \lVert F(0) \rVert / \beta$.
\end{lemma}
\begin{proof}
    This is a direct adaptation of \cite[Lemma 2.3]{Reisinger_Zhang_penalty}, where the same estimate is shown for discrete switching problems, relying on the monotonicity condition and the non-negativity of the penalty function. 
\end{proof}

We are now ready to show the well-posedness of the QVI \Cref{abstract_qvi_infinite}, which arises from the infinite horizon OCM \eqref{qvi_infinite}. In order to approximate the truncated problem, we require the following additional assumption.

\begin{ass}\label{ass_finite_dependence}
    For each $(n,x) \in \N \times \mathcal{X}$, there is a finite set $\Omega_{n,x}$ such that for every $r \in \R^A$, $F_a(n,x,r,\phi)$ only depends on the values of $\phi$ on $\Omega_{n,x}$, i.e.
    \begin{align*}
        F_a(n,x,r,\phi) = f_a(n,x,r,\phi\vert_{\Omega_{n,x}})
    \end{align*}
    for some function $f_a$.
\end{ass}

\begin{thm}
    \sloppy Given \Cref{ass_monotonicity,ass_finite_dependence}, the QVI \Cref{abstract_qvi_infinite} is well-posed. In particular, let $u^{\rho,N}$ denote the solution to the penalised problem \eqref{penalised_eq} with penalty parameter $\rho$ and truncation time $N$. Then, for fixed $c \geq 0$, ${u^N = \lim_{\rho\to\infty}u^{\rho,N}}$ exists and ${u = \lim_{N \to\infty} u^N}$ exists (here consider $u^N$ as functions from $\N \times \mathcal{X}$ to $\R^A$ by trivially defining $u^N_a(n,x) = 0$ for $n > N$). In both cases the convergence is monotone from below. Moreover, if $c>0$ then $u^{N}$ and $u$ solves \eqref{discrete_qvi} and \eqref{abstract_qvi_infinite} respectively.  
\end{thm}

\begin{proof}
    The existence of penalised solutions $u^{\rho,N}$ to \eqref{penalised_eq} can be established via a mollification argument, applying \cite[Theorem 2.5]{Reisinger_Zhang_penalty} to $G^{\rho,N}$. Uniqueness of the peanlised equation is given by \Cref{comparison_penalty}. The proof of the convergence of $u^{\rho,N} \uparrow u^N$, where $u^N$ solves \eqref{discrete_qvi} mirrors the argument in \cite[Theorem 2.6]{Reisinger_Zhang_penalty}: by \Cref{comparison_penalty}, for each $N$, $u^{\rho,N}$ is monotone in $\rho$, and together with the bound \Cref{lem_bound}, this gives the desired convergence. The continuity of the operators in the QVI is then used to show that $u^N$ solves \eqref{discrete_qvi}. We shall defer the reader to the relevant theorems cited above for the precise details.
    
    For the convergence of $u^N$ to some function $u$, we see also that $u^N$ is monotone in $N$: given a solution $u^{N+1}$ of \eqref{discrete_qvi}, we have $u^{N+1}(N,\cdot) \geq \mathcal{M}u^{N+1}(N,\cdot)$, so that $u^{N+1} \geq u^N$ by \Cref{comparison_qvi}. By \Cref{lem_bound}, $u^N$ is also bounded above, independent of $N$. Hence $u^N \uparrow u$. 
    
    \sloppy We now show that $u$ solves the QVI \eqref{abstract_qvi_infinite}. As $F(n,x,r,\phi)$ is only assumed to be continuous in $r$ in the untruncated case, we cannot directly appeal to continuity as before. We proceed as follows. Fix $n \in \N$, $x\in \mathcal{X}$, $a \in A$ and ${\varepsilon >0}$. As $\mathcal{M}$ is continuous, we have $u_a - \mathcal{M}u \geq 0$. By \Cref{ass_finite_dependence} and the pointwise convergence of $u^N$, we can choose a sufficiently large $\pr{N}$ such that
    \begin{align*}
        (u_a - u^{\pr{N}}_{a})\vert_{\Omega_{n,x}} &\leq \varepsilon,\\
        \pr{\varepsilon} = u_a(n,x) - u^{\pr{N}}_a(n,x) &\leq \varepsilon.
    \end{align*}
    Moreover, we can find $\tilde{N} \geq \pr{N}$ such that for all $b \neq a \in A$,
    \begin{align*}
        u_b(n,x) - u^{\pr{N}}_b(n,x) \leq \pr{\varepsilon},
    \end{align*}
    so that we can define $\tilde{u}$ by $\tilde{u}_a = u^{\pr{N}}_a(n,x)$ and $\tilde{u}_b = u^{\tilde{N}}_b(n,x)$ for $b \neq a$. Therefore,
    \begin{align*}
        u_a(n,x) - \tilde{u}_a = \max_{b\in A} (u_b(n,x) - \tilde{u}_b) = \pr{\varepsilon}.
    \end{align*}
    Then, we apply the monotonicity condition in \Cref{ass_monotonicity}, with $r = u(n,x) + (\varepsilon - \pr{\varepsilon})$, $s = \tilde{u}$, noting also that the functions $u_a - \varepsilon \leq u^{\tilde{N}}_a$ on $\Omega_{n,x}$. This gives
    \begin{align*}
        F_{a}(n,x,u(n,x) + \varepsilon - \pr{\varepsilon}, u_{a}) - F_{a}(n, x, \tilde{u}, u^{\tilde{N}}_{a}) \geq \beta \varepsilon.
    \end{align*}
    By the uniform continuity of $r \mapsto F_a(n,x,r,\phi)$, we have
    \begin{align*}
        F_{a}(n,x,u(n,x) + \varepsilon - \pr{\varepsilon}, u_{a}) &\geq F_{a}(n, x, u^{\tilde{N}}(n,x), u^{\tilde{N}}_{a}) - \omega\left(\lvert u^{\tilde{N}}_a(n,x) -u^{\pr{N}}_a(n,x) \rvert \right)  + \beta \varepsilon\\
        & \geq - \omega\left(\lvert u^{\tilde{N}}_a(n,x) -u^{\pr{N}}_a(n,x) \rvert \right)  + \beta \varepsilon,
    \end{align*}
    where the inequality follows from the fact that $u^{\tilde{N}}$ satisfies the QVI \eqref{discrete_qvi}, and $\omega$ is the modulus of continuity. Finally, by the convergence of $u^N$ to $u$, and the uniform continuity of ${r \mapsto F_a(n,x,r,\phi)}$ again, we have $F_{a}(n,x,u(n,x), u_{a}) \geq 0$.

    We now have $F_{a}(n,x,u(n,x), u_{a}) \geq 0$ and $u_a - \mathcal{M}u \geq 0$. Suppose $u_a - \mathcal{M}u > 0$. We show that $F_{a}(n,x,u(n,x), u_{a}) = 0$ to conclude. By the uniform continuity of ${r \mapsto F_a(n,x,r,\phi)}$, for any $\varepsilon >0$ there exists $\pr{N} \geq 0$ such that for all $N \geq \pr{N}$, 
    \begin{align*}
        0 \leq F_a(n,x,u(n,x),u_a) &< F_a(n,x,u^N(n,x), u_a) + \varepsilon\\
        & \leq  F_a(n,x,u^N(n,x), u^N_a) + \varepsilon,
    \end{align*}
    where the second inequality follows from the monotonicity condition. By the continuity of $\mathcal{M}$, $u^N_a - \mathcal{M}u^N > 0$ for sufficiently large $N$, and therefore $ F_a(n,x,u^N(n,x), u^N_a) = 0$. But as $\varepsilon$ is arbitrary, we have $F_{a}(n,x,u(n,x), u_{a}) = 0$ as required.    
\end{proof}

Thus, we have a straightforward computation scheme to solve for the OCM. We first set up the discrete QVIs arising from the problem and choose a suitable time for truncation, then the equations are approximated by the penalised problem. The solution of the penalised problem can then be in turn approximated iteratively with semismooth Newton methods \cite{Witte_Reisinger_penalty_obstacle}. Formally speaking, starting with an initialisation $v^{(0)}$ to the penalised problem
\begin{equation*}
   G^{\rho}(v) = \tilde{F}( v) - \rho\ \pi \left(\mathcal{M}v - v \right) = 0,
\end{equation*}
we obtain the next iterate by solving for
\begin{equation*}
    v^{(k+1)} = v^{(k)} - \mathcal{L}^{\rho}(v^{(k)})^{-1}G^{\rho}(v^{(k)}),
\end{equation*}
where $\mathcal{L}^{\rho}$ denotes the generalised derivative of the function $G^{\rho}$. This shall be the numerical scheme that we adopt for our numerical experiments in the next section.

\section{Numerical experiments}\label{section:numerical example}

In this section, we apply our observation cost framework to three numerical experiments. Sections \ref{subsection_random_walk} and \ref{subsection:HIV} analyse two infinite horizon problems. For these examples, we examine the numerical performance of the penalty method and Newton iterations, as well as the effects of the observation cost on the qualitative behavior of the solutions. For the penalised equations, we will employ the penalty function $\pi(x) = x^{+}$ as in \cite{Reisinger_Zhang_penalty}. Section \ref{subsection:bayesian} considers the parameter uncertainty formulation over a finite horizon. The solutions are obtained through backwards recursion from the terminal conditions. We examine the impact that the extra parameter uncertainty has on the optimal trajectories.

\subsection{Random walk with drift}\label{subsection_random_walk}

Consider an integer-valued random walk whose drift depends on the action space $A = \{+1, -1\}$. The probability of each step is parametrised by $\theta$. Specifically, for any $x \in \mathcal{X} = \mathbb{N}$,
\begin{align}\label{drift_prob}
    p(x+1 \mid x, +1) = \theta,\quad &  p(x-1 \mid x, +1) = 1- \theta; \nonumber\\
    p(x+1 \mid x, -1) = 1-\theta,\quad &  p(x-1 \mid x, -1) = \theta.
\end{align}
We also adopt the following reward function:
\begin{equation*}
    r(x,a) = \frac{1}{\vert x\vert+1}.
\end{equation*}
The mass of this reward function $r$ is concentrated around the origin, so naturally, the optimal action is one that reverts the process back towards the origin.

For this example, we consider the infinite horizon problem. Recall that the discrete QVI \eqref{qvi_infinite} reads: for all $n \geq 0$, $x \in \mathcal{X}$, and $a \in A$,
\begin{align} \label{qvi_random_walk}
    \min \bigg\{ & v^{n}_{a,x} - \gamma v^{n+1}_{a,x} - \Big(P^{n}_a\ r_a\Big)_x\ ,\  v^{n}_{a,x} - \bigg( P^{n}_a\ \overline{\gamma v^{1} + r}\bigg)_x  +  c_{\mathrm{obs}}   \bigg\} = 0, 
\end{align}
Note that there exists a path from $x$ to $y$ over $m$ units of time if and only if $m \geq \vert y-x \vert$ and $m \equiv y\ (\mathrm{mod}\ 2)$. If $\mathcal{S}^x_m$ denotes the set of states that can be reached from $x$ after $m$ units of time, then for a constant action, the $n$-step transition probabilties are given by
\begin{align*}
    p^{(n)}(x^{\prime}\mid x, +1) &= \begin{cases}
    \hfill \binom{n}{k} \theta^{k}(1-\theta)^{n-k} &,\ x^{\prime} \in \mathcal{S}^x_n;\\
    \hfill 0 &,\ x^{\prime} \notin \mathcal{S}^x_n,
    \end{cases} \\
    p^{(n)}(x^{\prime}\mid x, -1)  &= \begin{cases}
    \hfill \binom{n}{k} \theta^{n-k}(1-\theta)^{k} &,\ x^{\prime} \in \mathcal{S}^x_n;\\
    \hfill 0 &,\ x^{\prime} \notin \mathcal{S}^x_n,
    \end{cases}
\end{align*}
where $k =(n+x^{\prime}-x)/2$. For this problem, we shall also incorporate a switching cost $g$ whenever the drift changes (see \Cref{rem:switching}). Hence, in full, the QVI reads:
\begin{align}\label{qvi_explicit_rw}
    \min & \left\{  v^{n}_{+1,x} - \gamma v^{n+1}_{+1,x} - \sum_{x^{\prime}\in\mathcal{S}^x_n} \frac{1}{\vert x^{\prime} \vert +1} \binom{n}{k} \theta^k (1-\theta)^{n-k} ,  \nonumber \right. \\
    & \quad \left. v^{n}_{+1,x} -\gamma  \sum_{x^{\prime}\in\mathcal{S}^x_n} \binom{n}{k} \theta^k (1-\theta)^{n-k} \left(\frac{1}{\vert x^{\prime} \vert +1} +  \max\{v^{1}_{+1,x^{\prime}}, v^1_{-1, x^{\prime}} - g\} \right)   + c_{\mathrm{obs}}  \right\} = 0, \nonumber \\
    \min & \left\{  v^{n}_{-1,x} - \gamma v^{n+1}_{-1,x} - \sum_{x^{\prime}\in\mathcal{S}^x_n} \frac{1}{\vert x^{\prime} \vert +1} \binom{n}{k} \theta^{n-k} (1-\theta)^{k} ,  \nonumber \right. \\
    & \quad \left. v^{n}_{-1,x} - \gamma \sum_{x^{\prime}\in\mathcal{S}^x_n} \binom{n}{k} \theta^{n-k} (1-\theta)^{k} \left(\frac{1}{\vert x^{\prime} \vert +1} +  \max\{v^{1}_{+1,x^{\prime}} -g , v^1_{-1, x^{\prime}}\} \right)   + c_{\mathrm{obs}}  \right\} = 0. 
\end{align}

To close the system to ensure a unique solution, we enforce the following time and spatial boundary conditions. We impose a reflecting boundary at $x = \pm L$, where $L$ is suitably large. In particular,
\begin{align}\label{spatial_bound_rw}
    & p(L\mid L, +1) = \theta, \quad &&p(L-1 \mid L, +1)  = 1 - \theta, \nonumber\\
    &  p(L\mid L, -1)  = 1- \theta, \quad &&p(L-1\mid L, -1) = \theta, \nonumber\\
    & p(-L\mid -L, +1) = 1-\theta, \quad &&p( -L+1 \mid -L, +1) = \theta, \nonumber\\
    &  p(-L\mid -L, -1) = \theta, \quad &&p(-L+1\mid -L, -1) = 1- \theta,
\end{align}
so that the QVI \eqref{qvi_random_walk} for the states $ x = \pm L$ will use the transition probabilities \eqref{spatial_bound_rw} instead.

For the time boundary, we enforce an observation at some large time $N > 0$. The terminal condition then reads (for $-L < x < L$):
\begin{align}\label{time_bound_rw}
    \begin{dcases}v^{N}_{+1,x} - \gamma \sum_{x^{\prime}\in\mathcal{S}^x_N} \binom{N}{k^{\prime}} \theta^k (1-\theta)^{N-k^{\prime}} \left(\frac{1}{\vert x^{\prime} \vert +1} + \max\{v^{1}_{+1,x^{\prime}}, v^1_{-1, x^{\prime}} - g\}  \right)   + c_{\mathrm{obs}} = 0,\\
    v^{N}_{-1,x} - \gamma \sum_{x^{\prime}\in\mathcal{S}^x_N} \binom{N}{k^{\prime}} \theta^{N-k^{\prime}} (1-\theta)^{k^{\prime}} \left(\frac{1}{\vert x^{\prime} \vert +1} + \max\{v^{1}_{+1,x^{\prime}} -g , v^1_{-1, x^{\prime}}\} \right)    + c_{\mathrm{obs}} = 0.\end{dcases}
\end{align}
where $k^{\prime} = (N+x^{\prime}-x) /2$. The analogous equations hold for the spatial boundary $x=\pm L$, but with the transition probabilities \eqref{spatial_bound_rw}. These terminal conditions can be interpreted as the largest possible interval between two observations.

We now proceed to solve the penalised problem for the system \eqref{qvi_explicit_rw}, with boundary conditions \eqref{spatial_bound_rw} and \eqref{time_bound_rw}, through the use of semismooth Newton methods. To initialise the iteration, we solve for the uncoupled system
\begin{equation}\label{random_walk_initial_guess}
    \begin{dcases}v^{n}_{+1,x} - \gamma v^{n+1}_{+1,x} - \sum_{x^{\prime}\in\mathcal{S}^x_n} \frac{1}{\vert x^{\prime} \vert +1} \binom{n}{k} \theta^k (1-\theta)^{n-k} = 0, \\
    v^{n}_{-1,x} - \gamma v^{n+1}_{-1,x} - \sum_{x^{\prime}\in\mathcal{S}^x_n} \frac{1}{\vert x^{\prime} \vert +1} \binom{n}{k} \theta^{n-k} (1-\theta)^{k} = 0,\\
    0 \leq n < N,\ -L < x < L, \end{dcases}
\end{equation}
with the spatial boundary transition probabilities \eqref{spatial_bound_rw} and uncoupled time boundary condition
\begin{equation*}
    \begin{dcases}
    v^{N}_{+1,x} - \gamma \sum_{x^{\prime}\in\mathcal{S}^x_N} \binom{N}{k^{\prime}} \theta^{\pr{k}} (1-\theta)^{N-k^{\prime}} \left(\frac{1}{\vert x^{\prime} \vert +1} + v^{1}_{+1,x^{\prime}} \right)  + c_{\mathrm{obs}} = 0,\\
    v^{N}_{-1,x} - \gamma \sum_{x^{\prime}\in\mathcal{S}^x_N} \binom{N}{k^{\prime}} \theta^{N-k^{\prime}} (1-\theta)^{k^{\prime}}  \left(\frac{1}{\vert x^{\prime} \vert +1} + v^{1}_{-1,x^{\prime}} \right)  + c_{\mathrm{obs}}  = 0,\\
    -L < x < L. \end{dcases}
\end{equation*}

The system \eqref{random_walk_initial_guess} corresponds to the penalised equation with penalty parameter $\rho = 0$. The uncoupled time boundary condition is equivalent to enforcing an observation but with no switching (i.e., assuming that $\overline{v} = v_a$ in each  equation for $v_a$). The iteration terminates once a relative tolerance threshold of $10^{-8}$ is reached.

\begin{table}[t!]
    \centering
    \small
   \begin{tabular}{|c|c||c|c|c|c|c|c|}
        \hline
        & $\rho$ & $10^3$ & $2\times 10^3$ & $4 \times 10^3$
        & $8 \times 10^3$ & $16 \times 10^3$ & $32 \times 10^3$ \\
        \hline
        $c_{\mathrm{obs}} = 0$ & (a) & 2 & 2 & 2 & 2 & 2 & 2\\
        & (b)& 0.0063278 & 0.0031650 & 0.0015828 &  0.0007915 & 0.0003957 & 0.0001979\\
        \hline
        $c_{\mathrm{obs}} = 1/8$ & (a) & 5 & 5 & 5 & 5 & 5 & 5\\
        & (b)& 0.0048459 & 0.0024240 & 0.0012123  & 0.0006062 & 0.0003031 & 0.0001516\\
        \hline
        $c_{\mathrm{obs}} = 1/4$ & (a) & 6 & 6 & 6 & 6 & 6 & 6\\
        & (b)& 0.0033831 & 0.0016926 & 0.0008466 &  0.0004234 & 0.0002117 & 0.0001059\\
        \hline
        $c_{\mathrm{obs}}  = 1/2$& (a) & 6 & 6 & 6 & 6 & 6 & 6\\
        & (b)& 0.0015376 & 0.0007691 & 0.0003846 &  0.0001923 & 0.0000962 & 0.0000481\\
        \hline
        $c_{\mathrm{obs}}  = 1$& (a) & 7 & 7 & 7 & 7 & 7 & 7\\
        & (b)& 0.0006210 &  0.0003105 & 0.0001553 & 0.0000776 & 0.0000388 & 0.0000194\\
        \hline
        $c_{\mathrm{obs}}  = 2$ & (a) & 8 & 8 & 8 & 8 & 8 & 8\\
        & (b) & 0.0002077 & 0.0001038 & 0.0000519 & 0.0000260 & 0.0000130 & 0.0000065 \\
        \hline
        $c_{\mathrm{obs}}  = 4$ & (a) & 7 & 7 & 7 & 7 & 7 & 7\\
        & (b) & 0.0000852 & 0.0000426 & 0.0000213 & 0.0000157 & 0.0000053 & 0.0000027 \\
        \hline
        $c_{\mathrm{obs}}  = 6$ & (a) & 6 & 6 & 6 & 6 & 6 & 6\\
        & (b) & 0.0000307 & 0.0000154 & 0.0000077 & 0.0000038 & 0.0000019 & 0.0000010 \\
        \hline
   \end{tabular}
   \caption{Numerical results for the random walk with drift problem with switching cost $g=0$. Line (a): number of Newton iterations to reach the relative tolerance threshold of $1e^{-8}$. Line (b): the increment sizes $\lVert v^{\rho} - v^{2\rho} \rVert_{\infty}$.}
   \label{random_walk_error_table}
\end{table}

We investigate the numerical performance of our described methods for the case $\theta = 0.75$, ${\gamma = 0.99}$, $L=50$ and $N = 500$, across different cost parameters $c_{\mathrm{obs}}$. Computations are performed using MATLAB R2019b. The numerical solutions are shown in \Cref{random_walk_error_table}, for the case of zero switching cost. Row (a) shows that the number of Newton iterations required to reach the tolerance threshold is independent from the size of the penalty parameter $\rho$. Fewer iterations are required for more extreme values of $c_{\mathrm{obs}}$, but the overall number of iterations remains low across different observation costs. Row (b) shows the increments $\Vert v^{\rho} - v^{2\rho} \Vert_{\infty}$. The values suggests a first-order convergence of the penalisation error with respect to the penalty parameter $\rho$, which is in line with the analogous theoretical results in \cite[Theorem 3.9, 4.2]{Reisinger_Zhang_penalty}. Similar experiments with non-zero switching costs also demonstrate a first-order convergence.


\begin{figure}[ht!]
\centering
\begin{minipage}{.45\textwidth}
  \centering
  \includegraphics[width=7cm]{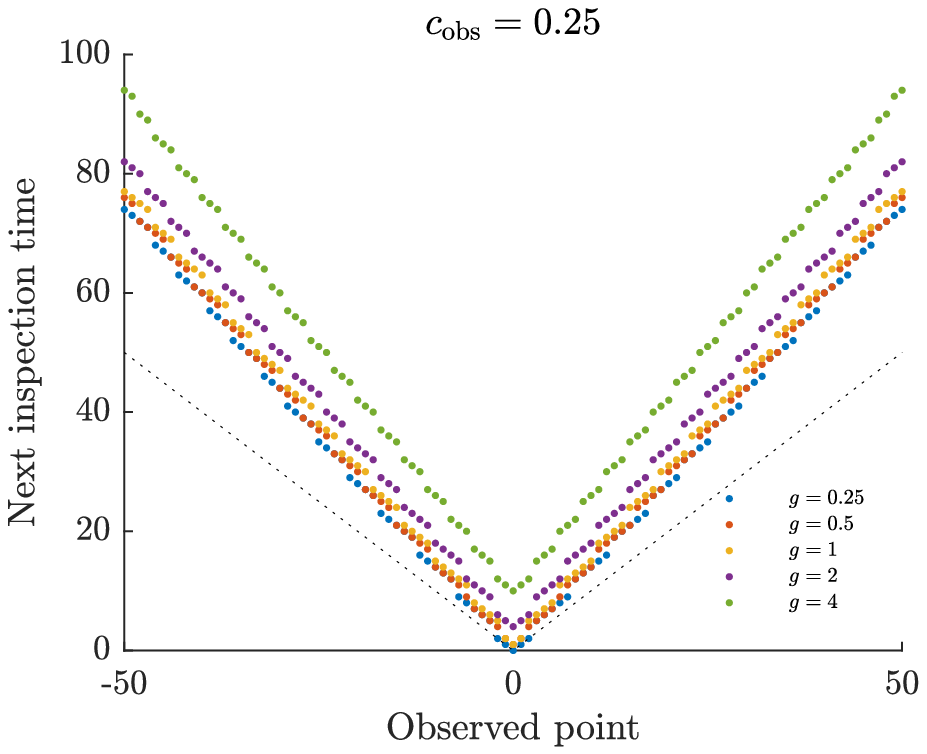}
  \end{minipage}
\begin{minipage}{.45\textwidth}
  \centering
  \includegraphics[width=7cm]{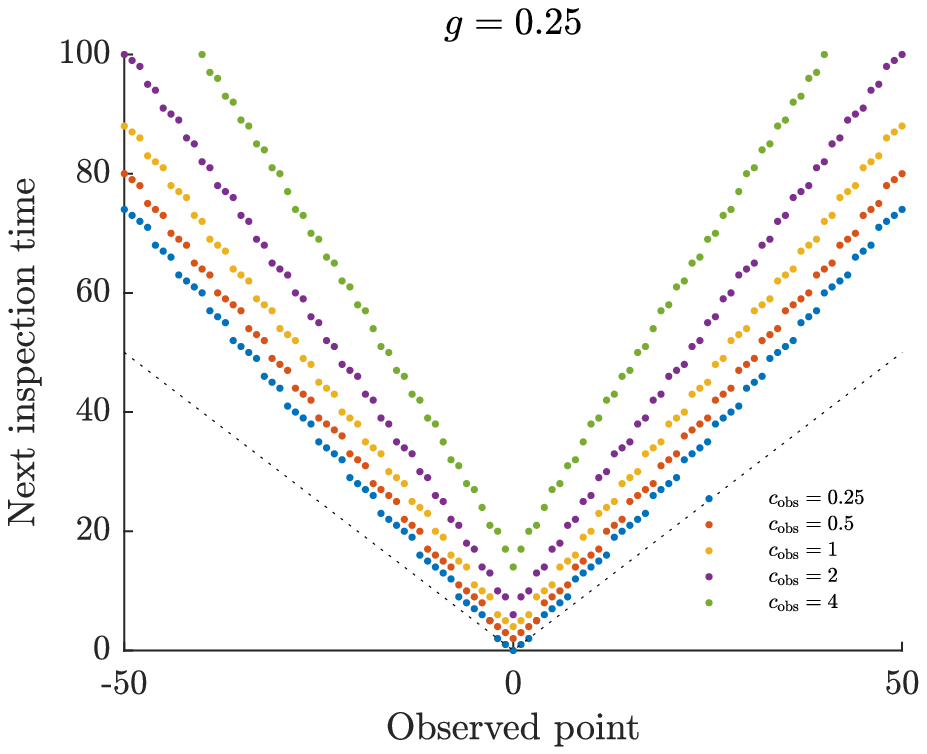}
\end{minipage}
\caption{The optimal waiting time before the next inspection, across the states $[-50, 50]$. Left: fixed observation cost $c_{\mathrm{obs}}=0.25$ against variable switching costs. Right: fixed switching cost $g=0.25$ against variable observation costs.}
\label{fixed costs vs switch}
\end{figure}

We now examine the qualitative behaviour of the solution. It is clear that if the chain is observed to be at a positive state, then the control should be switched to $a=-1$ for a negative drift and vice versa. \Cref{fixed costs vs switch} shows the optimal waiting time for the next inspection, across the state space of $[-50, 50]$. To demonstrate the extent to which the waiting time is due to the observation cost versus the switching cost, we vary both parameters and observe the optimal behaviour. The black dotted line represents the base case of zero observation and switching cost, in which case there is no need to observe until it is possible for the walk to cross the origin again. On the left plot, the observation cost $c_{\mathrm{obs}}$ is fixed whilst the switching cost  $g$ gradually increases, and vice versa on the right. We observe empirically that the observation cost contributes more significantly to the waiting times than the switching cost.

\subsection{Extension of an HIV-treatment model}\label{subsection:HIV}

In this subsection, we implement our formulation of the OCM to an HIV-treatment scheduling problem in \cite{winkelmann_markov_2014}. There, the authors modelled the problem with a continuous-time MDP with observation costs, but does not include the time elapsed variable in dynamic programming. This can be interpreted as an implicit assumption that the observer is given the state of the underlying process at initialisation. We shall implement a discretised version of their model under our formulation with the time elapsed variable. As alluded to in the introduction, this allows in addition initial conditions that are outdated or sub-optimal relative to the objective. We demonstrate the qualitative difference in the optimal policies when varying the initial conditions, whilst replicating the results in the original paper when the initial conditions coincide. We also examine the numerical performance of the penalty method when applied to the system of QVIs for this larger system, compared to that in Section \ref{subsection_random_walk}.

We now proceed to describe the original problem in \cite{winkelmann_markov_2014}. A continuous-time MDP is used to model virus levels of HIV-positive patients over time. With two types of treatment available, the action space is $A = \{0,1,2\}$ (where 0 represents no treatment given). Four virus strains are considered: WT denotes the wild type (susceptible to both treatments), R1 and R2 denotes strains that are each resistant to Treatment 1 and Treatment 2 respectively, and HR denotes the strain that is highly resistant to both. The level of each strain is represented by the states `none' (0), `low' ($l$), `medium' ($m$), and `high' ($h$). Therefore, the state space for the Markov chain is ${\mathcal{X} = \{0,l,m,h\}^4\ \cup \{\ast\} }$, where the asterisk represents patient death. Note in particular that $*$ is an absorbing state. The goal in the original model is to then minimise a cost functional $J: \mathcal{X} \times A \to \mathbb{R}$ of the form:
\begin{equation*}
    J(x, \alpha) = \E \left[ \sum^{\infty}_{j=0} \left( \int^{\tau_{j+1}}_{\tau_j} e^{-\gamma s} c(X_s, \iota(X_{\tau_j}))\ ds + e^{-\gamma \tau_{j+1}} c_{\mathrm{obs}} \right) \right],
\end{equation*}
where $\{\tau_j\}^{\infty}_{j=0}$ are the observation times, and the cost function $c: \mathcal{X} \times A \to \R$ is a linear combination of the productivity loss resulting from each patient's condition and their received treatment.

To adapt the model above for our formulation, we first discretise the MDP, taking each step to represent one day. We then take the model parameters from the original article \cite[Section 3]{winkelmann_markov_2014}, which provides the transition rate matrices $\{Q_a\}_{i \in A}$ and the cost function $c(x,a)$. The transition matrices $\{P_a\}_{a\in A}$ are then given by $P_a=e^{Q_a}$ (as the time unit in \cite{winkelmann_markov_2014} is one day). For illustration purposes, a sparse plot of the transition matrix $P_0$ is shown in \Cref{P0.jpg}. As we are considering maximisation problems in this article, we take $r = -c$ for the reward function. We can now formulate our problem in terms of the following QVI:
\begin{align*}
    \min \bigg\{ & v^n_{a,x} - \gamma v^{n+1}_{a,x} + \Big(e^{nQ_a}c_a \Big)_x\ ,\  v^{n}_{a,x} - \Big(e^{nQ_a}  \overline{\gamma  v^1 + c }\Big)_x + c_{\mathrm{obs}} \bigg\} = 0.
\end{align*}

\begin{figure}[ht!]
\centering
\includegraphics[width=15cm]{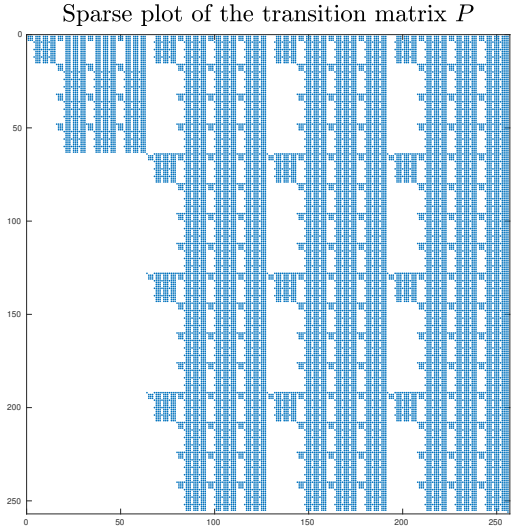}
\caption{Sparsity pattern of the transition matrix $P_0$ (the pattern is the same across all control states).  The state space is encoded as $\{1,\ldots,256\}$, by considering the state vectors [WT, R1, R2, HR] as a base-$4$ string in reverse order (for example, $[h, 0, l, l]$ corresponds to $83$). The death state $*$ is represented by $256$.}
\label{P0.jpg}
\end{figure}

We now follow the same procedure in \Cref{subsection_random_walk} to obtain a numerical solution. Note that for this problem, the spatial domain is finite and we also have a natural spatial boundary arising from the absorbing death state $\ast$, that is, for all $n \geq 0$ and $ a\in A$,
\begin{equation*}
    v^n_{a, \ast} = \sum^{\infty}_{k=0} l \gamma^k = \frac{l}{1-\gamma},
\end{equation*}
where $l$ is a constant representing the average GDP loss due to patient death, the value of which is taken from the parameters in \cite{winkelmann_markov_2014, winkelmann_phd}. A time boundary is once again enforced at some large time $N>0$, which can be interpreted as a mandatory observation at time $N$. Explicitly, this reads
\begin{equation*}
    v^{N}_{a,x} - \Big(e^{NQ_a}  \overline{\gamma  v^1 + c }\Big)_x + c_{\mathrm{obs}} = 0, \quad x \in \mathcal{X}\setminus\{\ast\},\ a \in A.
\end{equation*}

We now solve the associated penalised problem with semismooth Newton methods. As in \Cref{subsection_random_walk}, we choose the initial guess to be the solution to the penalised problem with $\rho = 0$, with uncoupled time boundary conditions. The iterations terminate once a relative tolerance threshold of $10^{-8}$ is reached. The numerical experiments are performed on MATLAB R2019b.

\Cref{HIV_convergence} shows the numerical solution for different values of the truncation time $N$ and $c_{\mathrm{obs}}$ across different penalty parameters $\rho$. Row (a) shows that the number of iterations remains constant with respect to $\rho$, much like the random walk experiment in \Cref{subsection_random_walk}. For this problem, the number of Newton iterations required to reach the $1\mathrm{e}{-8}$ threshold is higher at approximately 20 iterations. However, we find that convergence to the optimal policy is typically achieved within the first 2 iterations. This is depicted in \Cref{initial_guess}, which graphs the first two iterates as well as the final solution for the value function. Row (b) in \Cref{HIV_convergence} shows the successive increments $\lVert v^{\rho} - v^{2\rho} \rVert_{\infty}$ between doubling penalty values. Reassuringly, for this more complicated system, we still see a clear first-order convergence of the penalisation error with respect to the penalty parameter $\rho$. Even for small values of $\rho$, the successive increments were within $O(1)$ (in comparison to the magnitude of the solution which is of $O(10^6)$). This shows that the penalty approximation is very effective for small penalty parameters, and that it works well when extended to the class of QVIs that we introduced in \Cref{section_comparison}.

\begin{figure}[ht!]
\centering
  \includegraphics[width=9cm]{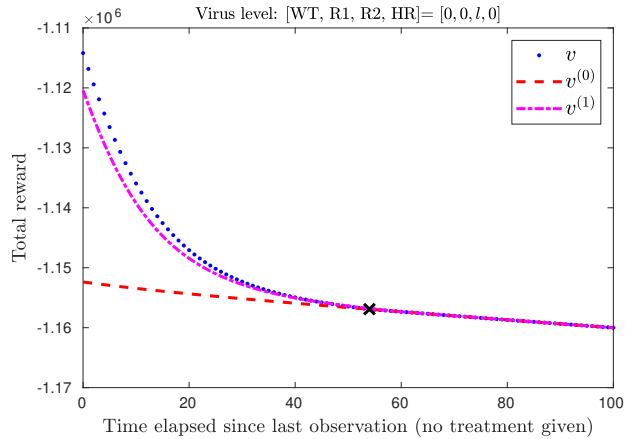}
\caption{Convergence of the Newton iterates towards the solution. The lines show the graphs of $n \mapsto v^n_{0, 4}$ for the initial guess $v^{(0)}$, first iterate $v^{(1)}$ and true solution $v$, where the state $[WT, R1, R2, HR] = [0, 0, l, 0]$ is encoded as $4$ in base $4$. The cross indicates the boundary between the observation regions.}
\label{initial_guess}
\end{figure}

\begin{table}[ht!]
    \centering
    \small
   \begin{tabular}{|c|c||c|c|c|c|c|c|}
   \hline
        & $\rho$ & $10^3$ & $2\times 10^3$ & $4 \times 10^3$
        & $8 \times 10^3$ & $16 \times 10^3$ & $32 \times 10^3$ \\
        \hline
        $N=150,\ c_{\mathrm{obs}} = 200$& (a) & 18 & 18 & 18 & 18 & 18 & 18\\
        & (b)& 1.6141 & 0.8071 & 0.4036 & 0.2018 & 0.1009 & 0.0504\\
        \hline
        $N=150,\ c_{\mathrm{obs}} = 400$& (a) & 21 & 21 & 21 & 21 & 21 & 21\\
        & (b)& 1.5147 & 0.7577 & 0.3790 & 0.1895 & 0.0948 & 0.0474\\
        \hline
        $N=150,\ c_{\mathrm{obs}} =800$ & (a) & 20 & 20 & 20 & 20 & 20 & 20\\
        & (b) & 1.4087 & 0.7047 & 0.3524 & 0.1762 & 0.0881 & 0.0441 \\
        \hline
        $N=300,\ c_{\mathrm{obs}} =200$& (a) & 20 & 20 & 20 & 20 & 20 & 20 \\
        & (b)& 1.6122 & 0.8061 & 0.4031 & 0.2015 & 0.1008 & 0.0504 \\
        \hline
        $N=300,\ c_{\mathrm{obs}} =400$ & (a) & 19 & 19 & 19 & 19 & 19 & 19\\
        & (b)& 1.5131 & 0.7569 & 0.3785 & 0.1893 & 0.0947 & 0.0473\\
        \hline
        $N=300,\ c_{\mathrm{obs}} =800$& (a) & 20 & 20 & 20 & 20 & 20 & 20 \\
        & (b)& 1.4102 & 0.7055 & 0.3528 & 0.1764 & 0.0882 & 0.0441 \\
        \hline
        $N=600,\ c_{\mathrm{obs}} =200$& (a) & 19 & 19 & 19 & 19 & 19 & 19 \\
        & (b)& 1.6111 & 0.8056 & 0.4028 & 0.2014 & 0.1007 & 0.0504 \\
        \hline
        $N=600,\ c_{\mathrm{obs}} =400$ & (a) & 17 & 17 & 17 & 17 & 17 & 17\\
        & (b) & 1.5114 & 0.7561 & 0.3781 & 0.1891 & 0.0945 & 0.0473 \\
        \hline
        $N=600,\ c_{\mathrm{obs}} =800$& (a) &  18 & 18 & 18 & 18 & 18 & 18\\
        & (b)& 1.4065 & 0.7036 & 0.3519 & 0.1760 & 0.0880 & 0.0440 \\
        \hline
   \end{tabular}
   \caption{Numerical results for the HIV-treatment problem. Line (a): number of Newton iterations. Line (b): the increments $\lVert v^{\rho} - v^{2\rho} \rVert$.}
    \label{HIV_convergence}
\end{table}

We now analyse the behaviour of the value function when plotted as a function against time. The top-left graph of \Cref{HIV_optimal_vs_sub} depicts an instance where the patient is under a stable condition. Here the observation region is $[15, N]$. There are limited benefits of frequently paying a high observation cost when it is unlikely that the patient's condition will deteriorate over a short period of time. On the other hand, the top-right graph has an observation region of $[0,53]$. The mathematical intuition behind this is that beyond the observation region, the MDP is expected to enter the absorbing state $*$ with high probability, and the negative reward associated with this absorbing state outweighs any potential benefits of paying the observation cost $c_{\mathrm{obs}}$ for information. In the original model in \cite{winkelmann_markov_2014}, one determines the optimal policy based on an immediate observation in hand. Putting this in the context of our formulation, this amounts to fixing an initial condition in the form of $y= (1,x,a) \in \mathcal{Y}$, and looking forward ahead in time to find the first observation time. This overlooks situations such as that occurring in the top-right graph of \Cref{HIV_optimal_vs_sub}: by initialising at the time origin, one immediately `loops back' for an immediate observation, and therefore does not see the effect of the passage of time on the optimal observation policy.

To examine the behaviour around the decision boundaries, we plot the central finite difference terms $(v^{n+1}_{a,x} - v^{n-1}_{a,x})/2\Delta n$ in the bottom row of \Cref{HIV_optimal_vs_sub}, underneath their respective graphs of the value function. If we consider the plots as a discretisation of a continuous value function, we see that there is much bigger variation within the observation region. Critically, there is non-smoothness across the boundary in the bottom-left graph. This suggests that the solution in continuous-time is $C^2$ in time within each decision region, but only $C^1$ across the boundary. This is in line with theoretical results on the regularity of viscosity solutions in optimal stopping and switching problems \cite[Chapter 5]{Pham2009}, which is a potential direction for future analysis.

\begin{figure}[ht!]
\centering
\begin{minipage}{.45\textwidth}
  \centering
  \includegraphics[width=7cm]{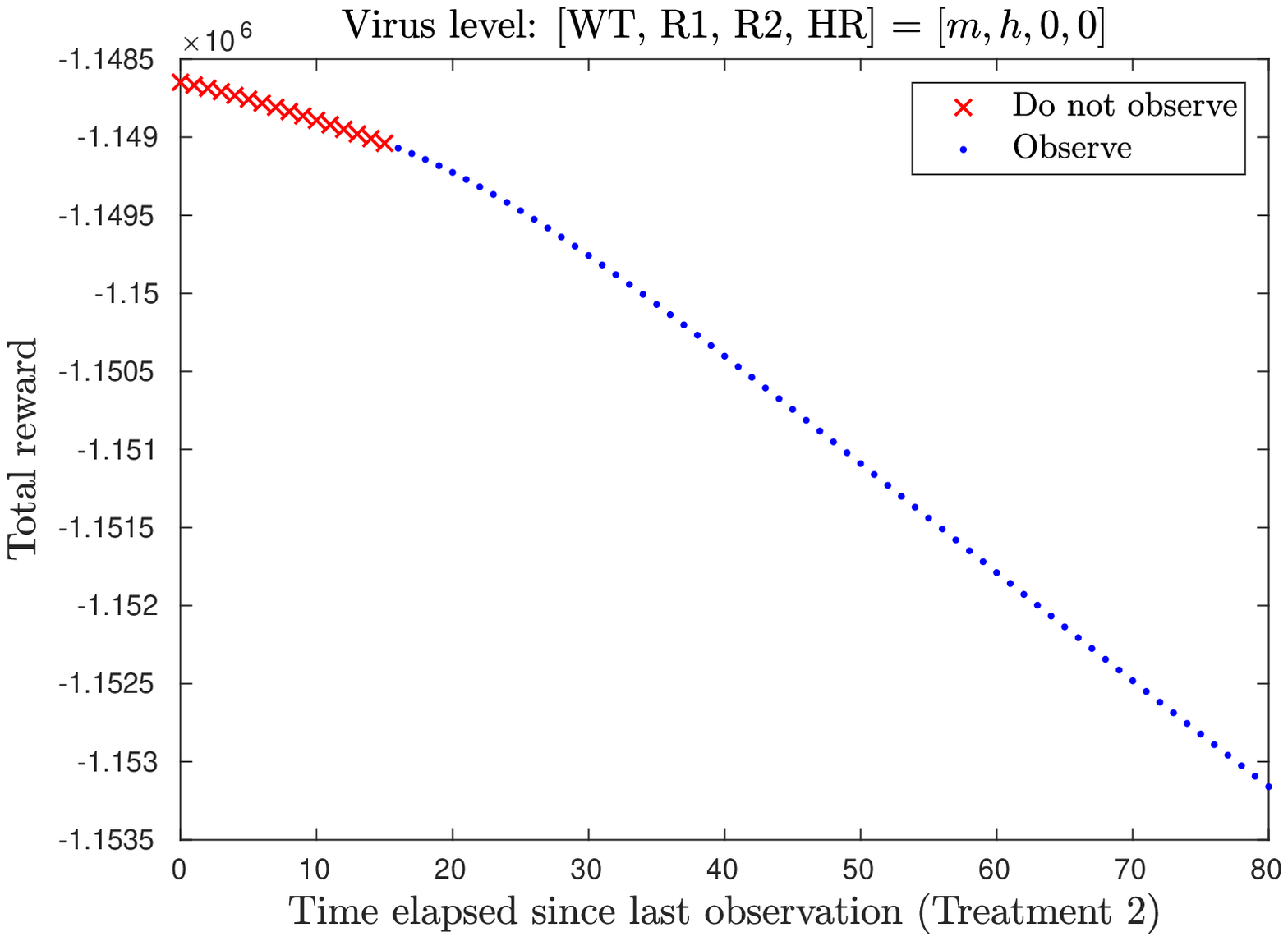}
\end{minipage}
\begin{minipage}{.45\textwidth}
  \centering
  \includegraphics[width=7cm]{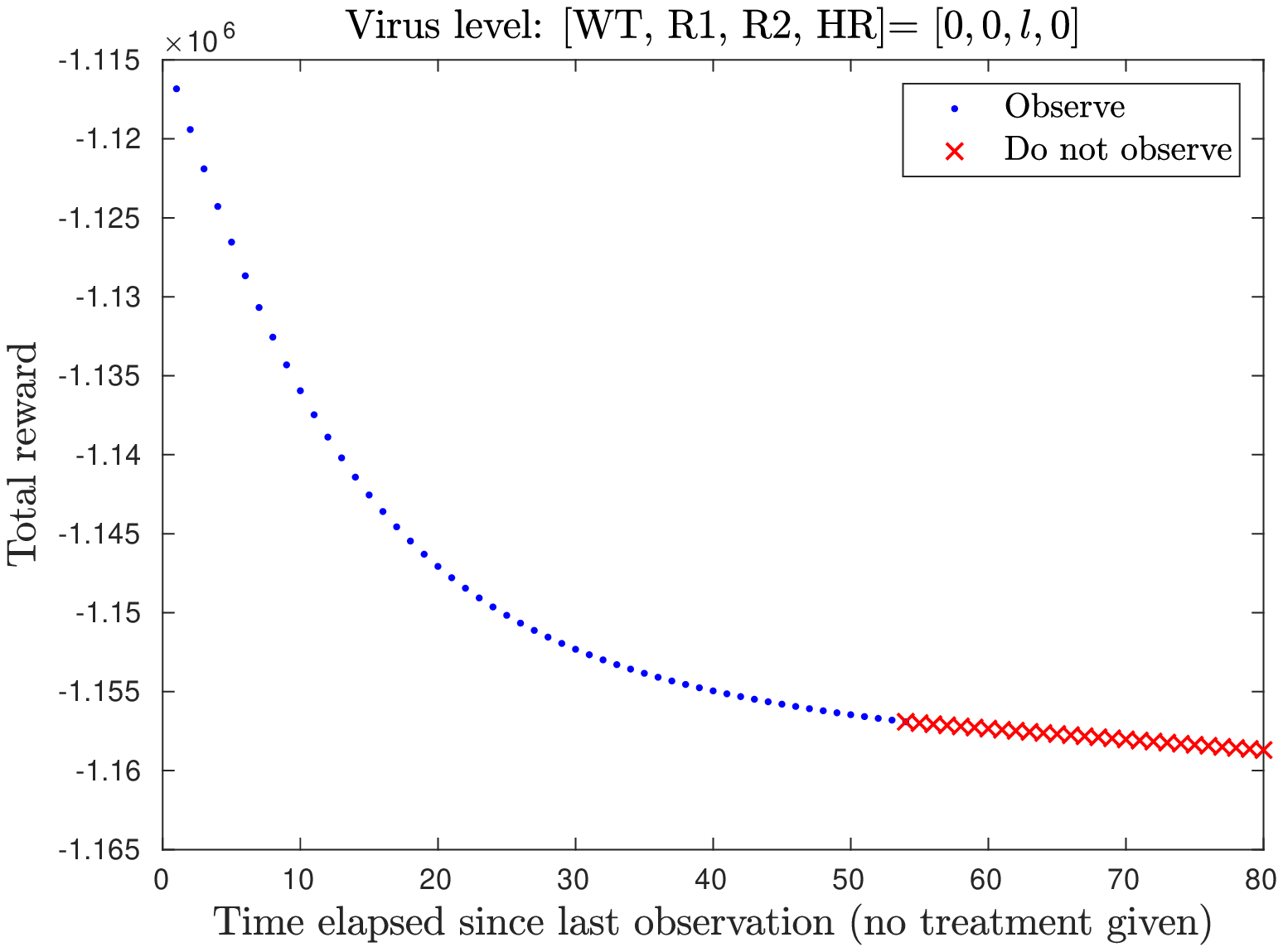}
\end{minipage}
\begin{minipage}{.45\textwidth}
  \centering
  \includegraphics[width=7cm]{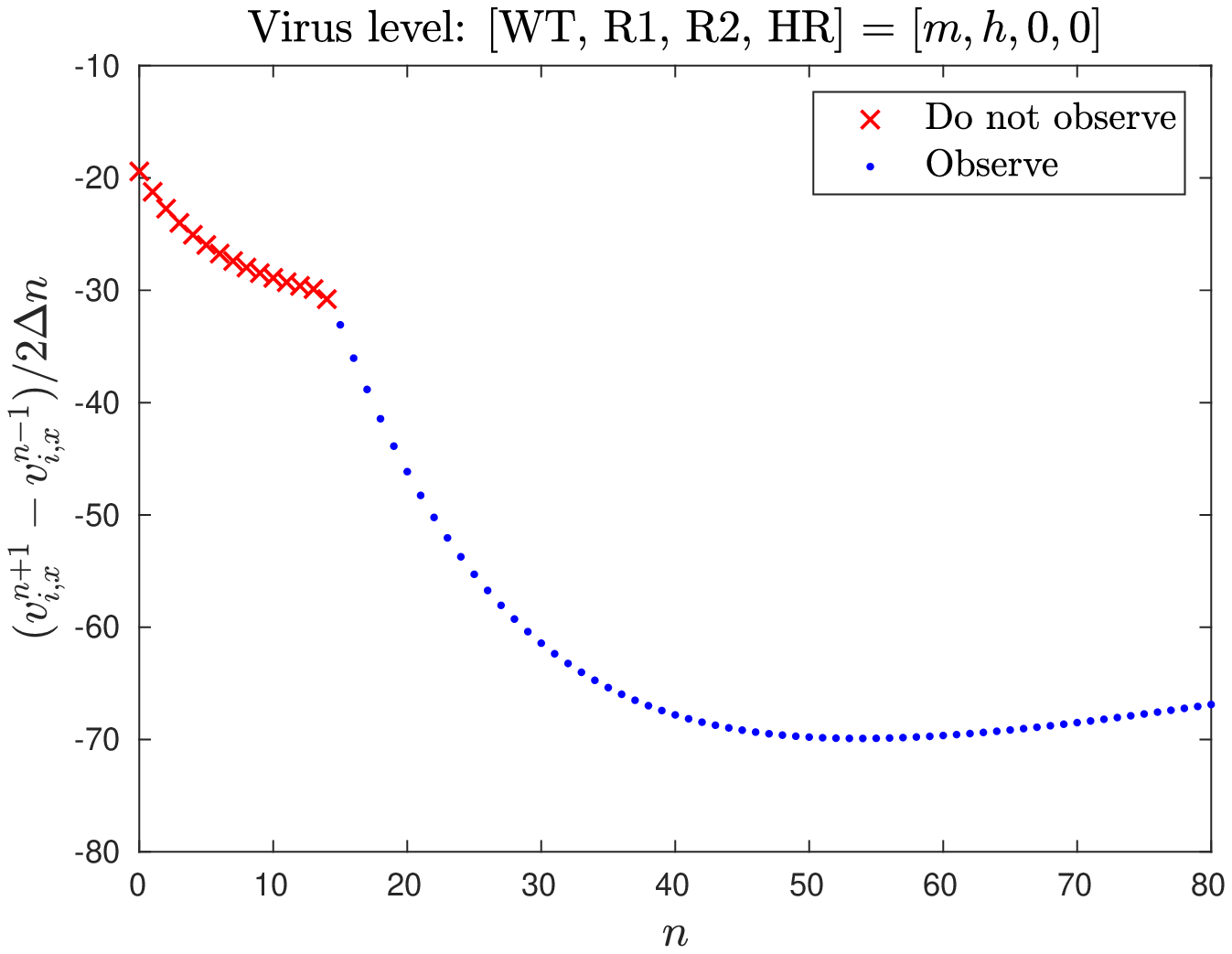}
\end{minipage}
\begin{minipage}{.45\textwidth}
  \centering
  \includegraphics[width=7cm]{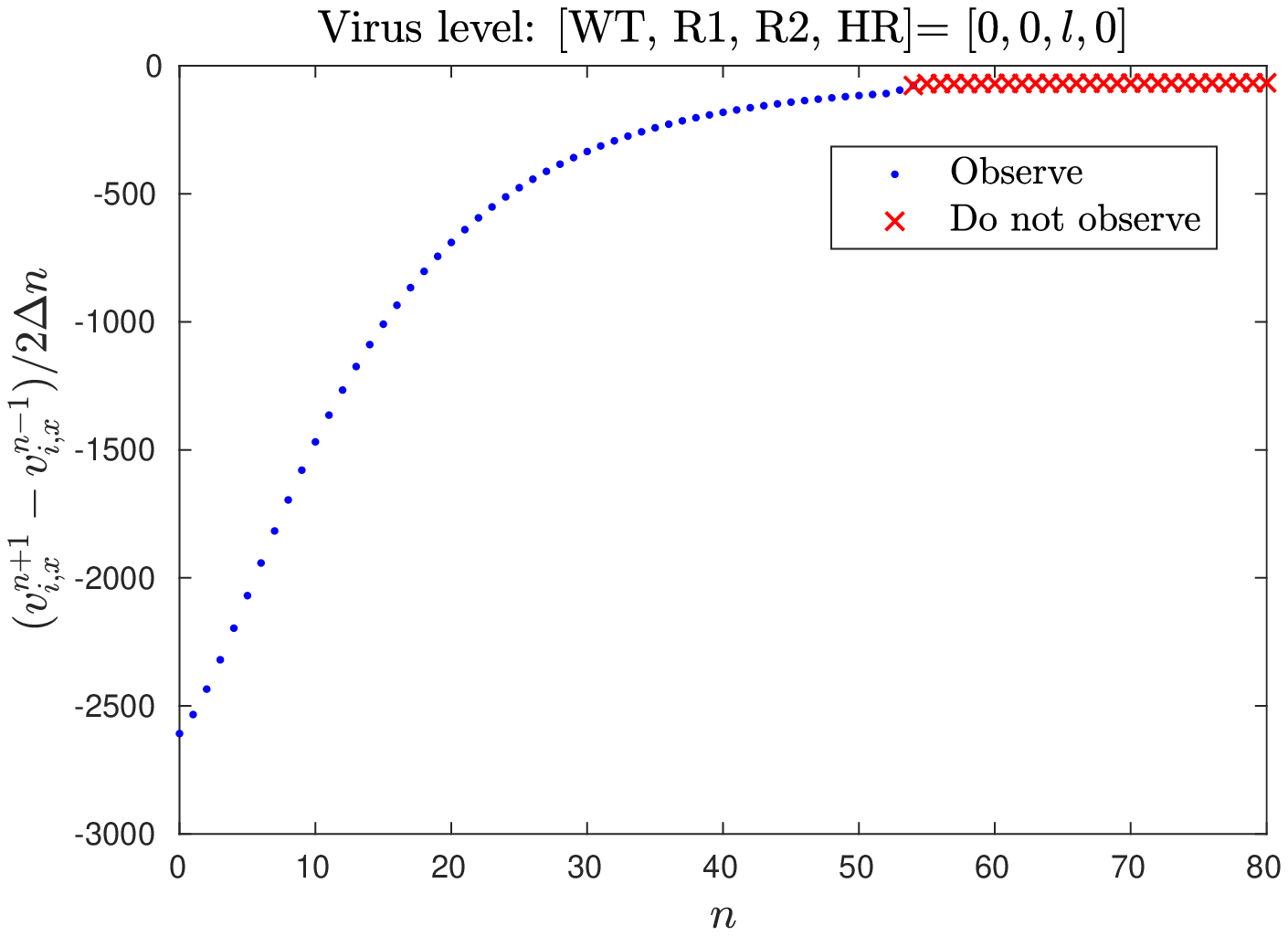}
\end{minipage}
\caption{The value function exhibits two qualitatively different decay modes depending on the starting states $x$. Left: a stable condition with the correct treatment. Right: a worse condition with no treatment. The top row shows the mappings $n\mapsto v^{n}_{i,x}$. The bottom row plots the corresponding central finite difference terms.}
\label{HIV_optimal_vs_sub}
\end{figure}

\section{Observation cost with parameter uncertainty}\label{section_parameter_uncertainty}

For this final section, we consider an extrension of the OCM with parameter uncertainty in the dynamics of the Markov chain. We shall adopt the approach of Bayesian adaptive control. Suppose that the transition kernel $p$ now depends on an unknown parameter $\theta \in \Theta$, where $\Theta$ denotes a finite parameter space. We write $p_{\theta}(\cdot \mid x ,a)$ for a fixed value of $\theta$. To consider a Markov system for the problem, we take $\mathcal{X} \times \Theta$ as our underlying space, with transition kernel
\begin{align*}
    \mathfrak{p}((x^{\prime}, \theta^{\prime}) \mid (x, \theta), a) \coloneqq \mathbbm{1}_{\{\theta = \theta^{\prime}\}} p_{\theta}(x^{\prime} \mid x ,a),\quad (x, \theta), (x^{\prime}, \theta^{\prime}) \in \mathcal{X} \times \Theta,\ a \in A.
\end{align*}
For each value of $\theta \in \Theta$, we associate a probability measure $p^{\theta}_0 \in \Delta_{\mathcal{X}}$, so that the initial distribution is given by
\begin{align*}
    \mathfrak{p}_0(x, \theta) = \sum_{\theta \in \Theta} p^{\theta}_0(x) \rho_0(\theta)
\end{align*}
for some $\rho_0 \in \Delta_{\Theta}$. Given the half-step construction as laid out in the beginning of this section, the OCM with parameter uncertainty can be written as a POMDP $\langle \mathcal{X} \times \Theta, \mathcal{X}_{\varnothing}, \mathcal{A}, \mathfrak{p}, \mathfrak{p}_0, \mathfrak{q}, \mathfrak{q}_0, r_{\mathrm{obs}} \rangle$ over the timescale $\frac{1}{2}\N$, where the domain of the transition kernel $\mathfrak{p}$ is extended to $(\mathcal{X} \times \Theta) \times \mathcal{A}$ by defining
    \begin{align*}
        \mathfrak{p}(\cdot \mid x, \theta, i) = \delta_{(x, \theta)}(\cdot),\quad i \in \mathcal{I},
    \end{align*}
and the observation kernel $\mathfrak{q}$ is now defined on $\mathcal{A} \times (\mathcal{X}\times \Theta)$, with
     \begin{align*}
        \mathfrak{q}(\cdot \mid a, x, \theta) &= \delta_{\varnothing}(\cdot),  \quad a \in A, \\
         \mathfrak{q}(\cdot \mid i, x, \theta) &= i\cdot\delta_{x}(\cdot) + (1-i)\delta_{\varnothing}(\cdot), \quad i \in \mathcal{I}.     
     \end{align*}
The initial observation kernel $\mathfrak{q}_0$ will be taken as $\mathfrak{q}_0(\cdot\mid x, \theta) = \delta_x(\cdot)$ (see \Cref{rem:initial_kernels}). The set of admissible policies, denoted by $\Pi^{\Theta}_{\mathrm{obs}}$ in this case, can be established analogously as in \Cref{defn:policy_ocm}. Denote the canonical measure here by $\PP^{\pi}_{\mathfrak{p}_0}$, under which $\theta$ can be considered as a constant process, i.e. $\theta_{n+1} \equiv \theta_n$ with $\theta_0 \sim \rho_0 \in \Delta_{\Theta}$, and $\rho_0$ can be interpreted as a prior estimate for $\theta$. 

When considering the belief MDP for this problem, the observable sequence at time $t$ remains as previously, 
\begin{align*}
    h_t = (\bar{x}_0, \pi_0, \ldots, \bar{x}_{t-1/2}, \pi_{t-1/2}, \bar{x}_t) \in \mathcal{H}_t.
\end{align*}
Then, the belief state $\PP^{\pi}_{\mathfrak{p}_0}(x_t, \theta_t \mid h_t)$ can be decomposed as
\begin{align}\label{eq:belief_param_decompose}
    \PP^{\pi}_{\mathfrak{p}_0}(x_t, \theta_t \mid h_t) & = \sum_{\theta \in \Theta} \PP^{\pi}_{\mathfrak{p}_0}(x_t \mid \theta_t , h_t)\ \PP^{\pi}_{p_0}(\theta_t \mid h_t).
\end{align}
For a fixed value of $\theta$, ${\PP^{\pi}_{\mathfrak{p}_0}(x_t \mid \theta_t , h_t)}$ has a finite dimensional characterisation by the Markov property. As in the previous section, we denote this characterisation by $y=(y_t)_t$, which is a tuple given by
\begin{align*}
    y_t \coloneqq 
    \begin{cases}
    (k, x_k, a_k),&\ \mbox{if $k = \argmax_{n\leq t} \{i_{n} = 1 \} \neq t$ },\\
    (t, x_t, \varnothing),&\ \mbox{otherwise}.
    \end{cases}
\end{align*}
The second term on the right hand side of \eqref{eq:belief_param_decompose}, $\PP^{\pi}_{\mathfrak{p}_0}(\theta_t \mid h_t)$ can be interpreted as the posterior distribution of $\theta$ at time $t$. Denote this term by $\rho_t(\theta_t)$. Note that $\rho_t$ can be computed online via the classical Bayes' Theorem,
\begin{align}\label{bayes}
    \rho_{t}(\theta) & = \frac{\PP^{\pi}_{\mathfrak{p}_0}(\bar{x}_{t} \mid \theta, h_{t-\frac{1}{2}}, \pi_{t-\frac{1}{2}})}{\sum_{\pr{\theta} \in \Theta}\PP^{\pi}_{\mathfrak{p}_0}(\bar{x}_{t} \mid \pr{\theta}, h_{t-\frac{1}{2}}, \pi_{t-\frac{1}{2}})\ \rho_{t-\frac{1}{2}}(\pr{\theta})}\ \rho_{t-\frac{1}{2}}(\theta).
\end{align}
In the OCM, observations only occur at the half steps, therefore we have in fact $\rho_{n} \equiv \rho_{n-1/2}$ for $n \in \N$. Thus, the update \eqref{bayes} can be reduced to
\begin{align}
    \rho_{n}(\theta) & = \frac{\PP^{\pi}_{\mathfrak{p}_0}(\bar{x}_{n-\frac{1}{2}} \mid \theta, y_{n-1}, i_{n-1})}{\sum_{\pr{\theta} \in \Theta}\PP^{\pi}_{\mathfrak{p}_0}(\bar{x}_{n-\frac{1}{2}} \mid \pr{\theta}, y_{n-1}, i_{n-1})\ \rho_{n-1}(\pr{\theta})}\ \rho_{n-1}(\theta) \nonumber \\
    & \eqqcolon U(\rho_{n-1}, y_{n-1}, i_{n-1}), \quad n \in \N.\label{bayes:integer}
\end{align}
The belief state at time $t$ can now be represented by $ y_t \in \mathcal{Y}$ and $\rho_t \in \Delta_{\Theta}$, with its transitions given by the kernel $\pr{\mathfrak{p}}= (\pr{\mathfrak{p}}_t)_{t \in \frac{1}{2}\N}$ on $\mathcal{Y} \times \Delta_{\Theta}$, given $(\mathcal{Y} \times \Delta_{\Theta})  \times\mathcal{A}$:
\begin{align*}
    \pr{\mathfrak{p}}_n( \pr{y}, \pr{\rho} \mid y, \rho, i) &= i \cdot \sum_{\theta \in \Theta}p^{(n-k)}_{\theta}(\hat{x} \mid x, a)\rho(\theta)\ \mathbbm{1}_{\{\pr{y} = (n, \hat{x}, \varnothing),\ \pr{\rho} = U(\rho, y, i) \}}\\
    &\qquad + (1-i) \mathbbm{1}_{\{\hat{y}= y,\ \pr{\rho} = \rho\}},\ && y = (k,x,a),\\
    \pr{\mathfrak{p}}_{n + \frac{1}{2}}(\pr{y}, \pr{\rho} \mid y, \rho, a^{\prime}) &= \mathbbm{1}_{\{\pr{y} = (k, x, a^{\prime}),\ \pr{a} = a,\ \pr{\rho} = \rho\}},\ && y = (k,x,a),\\
    \pr{\mathfrak{p}}_{n + \frac{1}{2}}(\pr{y}, \pr{\rho} \mid y, \rho, a^{\prime}) &= \mathbbm{1}_{\{\hat{y} = (n, x, a^{\prime}),\ \pr{\rho} = \rho\}},\ && y = (n ,x ,\varnothing).
\end{align*}
As before, we shall not distinguish between policies for the POMDP and policies for the belief state MDP. For the finite horizon problem, let $y=(k,x,a) \in \mathcal{Y}$, $\rho \in \Delta_{\Theta}$, and consider
\begin{equation}\label{reward_func_bayes}
    J(t, y, \rho, \pi) = \E^{\pi}_{y,\rho}\left[ \sum^{2N}_{n=t} r^{\rho}_{n/2}\left(y_{\frac{n}{2}}, \pi_{\frac{n}{2}}\right) \right],\quad 0 \leq t \leq 2N,\ \pi \in \Pi_{\mathrm{obs}},
\end{equation}
where for $y = (k,x,a)$, $i \in \mathcal{I}$, $a^{\prime} \in A$, $n \in \N$,
    \begin{align*}
    r^{\rho}_n(y, i) &= - i \cdot c_{\mathrm{obs}},\\
    r^{\rho}_{n+\frac{1}{2}}(y, a^{\prime})& = \sum_{\theta \in \Theta}\sum_{x^{\prime} \in \mathcal{X}} r(x^{\prime}, a^{\prime}) p^{(n-k)}_{\theta}(x^{\prime} \mid x, a) \rho (\theta),
    \end{align*}
    and for $y = (n,x,\varnothing)$,
    \begin{align*}
        r^{\rho}_{n+\frac{1}{2}}(y, a^{\prime})& = r(x, a^{\prime}).
    \end{align*}
The value function is
\begin{equation}\label{value_bayes}
    v(t, y, \rho) = \sup_{\pi \in \Pi_{\mathrm{obs}}} J(t, y,\rho, \pi).
\end{equation}

As in the previous case without parameter uncertainty, the dynamic programming equation can be reduced to only the integer time steps. We state the optimality equation for the observation cost model under parameter uncertainty below.

\begin{prop}For $y=(k,x,a) \in \mathcal{Y}$ and $\rho \in \Delta_{\Theta}$, the value function \eqref{value_bayes} satisfies the following equation:
\begin{align}\label{finite_dpp_bayes}
    v(n, (k,x,a), \rho) = &\max \Bigg\{ v(n+1,  (k,x,a), \rho)  +  \sum_{\substack{\theta \in \Theta \\x^{\prime} \in \mathcal{X}}} p^{(n-k)}_{\theta}(x^{\prime} \mid x, a) r(x^{\prime},a)\rho(\theta) , \nonumber\\
    &   \sum_{\substack{\theta \in \Theta \\x^{\prime} \in \mathcal{X} }} p^{(n-k)}_{\theta}(x^{\prime} \mid x, a) \rho(\theta) \  \Big[   \max_{a^{\prime} \in A} \Big(v(n+1, (n, x^{\prime}, a^{\prime}), \rho^{\prime}) + r(x^{\prime},a^{\prime})\Big)  \Big] - c_{\mathrm{obs}} \Bigg\},
\end{align}
where $\rho^{\prime}=U(\rho, y, 1) $ as in \eqref{bayes:integer}.
\end{prop}

For the infinite horizon case, a similar stationary argument leads us to the objective function and value function:
\begin{align}
    J(y, \rho, \pi) &= \E^{\pi}\left[ \sum^{\infty}_{n=0} \gamma^{\floor{\frac{n}{2}}} r^{\rho}\left(y_{\frac{n}{2}}, \pi_{\frac{n}{2}}\right) \right],\ y \in \mathcal{Y},\ \rho \in \Delta_{\Theta},\ \pi \in \Pi_{\mathrm{obs}}, \nonumber \\
    v(y, \rho) &= \sup_{\pi \in \Pi_{\mathrm{obs}}} J(y,\rho, \pi). \label{value_bayes_inf}
\end{align}

\begin{prop}\label{prop_infty_bayes}
For $y=(n,x,a) \in \mathcal{Y}$ and $\rho \in \Delta_{\Theta}$, the value function \eqref{value_bayes_inf} satisfies the following equation:
\begin{align}\label{infty_horizon_bayes}
    v((n,x,a), \rho) = &\max \Bigg\{  \gamma v((n+1,x,a), \rho) + \sum_{\substack{\theta \in \Theta \\x^{\prime} \in \mathcal{X}}} p^{(n)}_{\theta}(x^{\prime} \mid x, a) r(x^{\prime},a) \rho(\theta), \nonumber\\
    & \sum_{\substack{\theta \in \Theta \\x^{\prime} \in \mathcal{X}}} p^{(n)}_{\theta}(x^{\prime} \mid x, a) \rho(\theta) \  \Big[  \max_{a^{\prime} \in A} \Big( \gamma v((1, x^{\prime}, a^{\prime}), \rho^{\prime}) + r(x^{\prime},a^{\prime}) \Big) \Big]  - c_{\mathrm{obs}}  \Bigg\},
\end{align}
where $\rho^{\prime}=U(\rho, y, 1) $ as in \eqref{bayes:integer}.
\end{prop}

It is worth noting that both \eqref{finite_dpp_bayes} and \eqref{infty_horizon_bayes} are MDPs over the augmented space ${\mathcal{Y} \times \Delta_{\Theta}}$. The inclusion of the simplex $\Delta_{\Theta}$ makes the MDP non-discrete. For computation, one would have to approximate the solution, either via computing a discrete MDP on a finite grid on ${\mathcal{Y} \times \Delta_{\Theta}}$, or via functional approximation methods such as the use of neural networks on larger scale problems. We refer the reader to the textbook \cite{Kushner_dupuis} and survey paper \cite{rust1996numerical} for a comprehensive review of choosing appropriate approximating grids. Then, given a finite grid $\mathbb{G} = \{s_1, \ldots ,s_G\}$ on the simplex $\Delta_{\Theta}$, one can define the approximating transition kernels by
\begin{align*}
    \mathfrak{p}_{\mathbb{G}}(\pr{y}, s_i \mid y, s_j, \pi) = \frac{\pr{\mathfrak{p}}(\pr{y}, s_i \mid y, s_j,a)}{\sum_{i=1}^G \pr{\mathfrak{p}}(\pr{y}, s_i \mid y, s_j, a)},\quad y, \pr{y} \in \mathcal{Y},\ s_i, s_j \in \mathbb{G},\ \pi \in \mathcal{A},
\end{align*}
so that one solves the approximating finite MDP on $\mathcal{Y} \times \mathbb{G}$ instead. For the QVIs resulting from the observation cost problems, we can solve the MDPs by penalisation, as detailed in \Cref{section_comparison}.
\subsection{Random walk with drift with parameter uncertainty}\label{subsection:bayesian}

In this subsection, we consider a random walk with drift, as set up in \Cref{subsection_random_walk}, but with the additional assumption that the true value of the drift parameter $\theta$ is unknown to the user. To avoid complications with boundary conditions and infinite domains, we shall only consider the finite horizon problem. Recall that for a fixed value of $\theta$ and constant action, the $n$-step transition probabilities are given by
\begin{align*}
    p^{(n)}_{\theta}(x^{\prime}\mid x, +1) = \binom{n}{k} \theta^k (1-\theta)^{n-k}, \
    p^{(n)}_{\theta}(x^{\prime}\mid x, -1)  = \binom{n}{k} \theta^{n-k} (1-\theta)^{k},\quad \pr{x} \in \mathcal{S}^x_n,
\end{align*}
where $\mathcal{S}^x_n$ is the set of states that can be reached from $x$ after $n$ units of time, and $k = \frac{1}{2}(n +x^{\prime}-x)$. We choose the prior from a family of beta distributions to obtain conjugacy in the parameter distributions. The posterior can be updated as follows. Suppose the prior $\rho_0 \sim \mathrm{Beta}(\alpha,\beta)$ and the next observation occurs at time $n$ at a state $\pr{x} \in \mathcal{X}$. Then a standard calculation shows that
\begin{align*}
    \int_{\Theta} p^{(n)}_{\theta}(x^{\prime}\mid x, +1)\ \rho_0(d\theta)  = g(k \mid n, \alpha, \beta), \
    \int_{\Theta} p^{(n)}_{\theta}(x^{\prime}\mid x, -1)\ \rho_0(d\theta)  = g(n-k \mid n, \alpha, \beta),
\end{align*}
where
\begin{equation*}
    g(k \mid n, \alpha, \beta) = \binom{n}{k} \frac{B(k + \alpha, n -k + \beta)}{B(\alpha,\beta)},\ k = \frac{1}{2}(n + \pr{x} - x),
\end{equation*}
and $g$ is the probability mass function of the Beta-binomial distribution, $B(\alpha,\beta)$ is the Beta function. The posterior distribution is then given by
\begin{align*}
   \rho_n \sim \begin{cases}
   \mathrm{Beta}(\alpha + k, \beta + n - k), & a_{0} = + 1,\\
   \mathrm{Beta}(\alpha + n - k, \beta + k), & a_{0} = - 1.\\
   \end{cases}
\end{align*}

The distributions $\{\rho_n\}$ can then be described by a finite number of values over a finite time horizon. Hence we write $ v(n,(k,x,a),(\alpha,\beta))$ for $v(n,y,\rho)$. Let us first consider the same reward function $r(x,a) = r(x) = \frac{1}{\lvert x \rvert + 1}$ as in the previous section. We can write the QVI \eqref{finite_dpp_bayes} as
\begin{align*}
    & v(n, (k,x,a), (\alpha, \beta)) \nonumber \\
    =\ &\max \Bigg\{ v(n+1,  (k,x,a),  (\alpha, \beta))  +  \sum_{x^{\prime} \in \mathcal{X}} \pr{g}(\pr{x}, x \mid n-k, \alpha, \beta)\ r(\pr{x}), \nonumber \\
    &\   \sum_{x^{\prime} \in \mathcal{X}}  \pr{g}(\pr{x}, x \mid n-k, \alpha, \beta)  \Big[   \max_{a^{\prime} \in A} \Big(v(n+1, (n, x^{\prime}, a^{\prime}), (\pr{\alpha}, \pr{\beta})) + r(x^{\prime})\Big)  \Big] - c_{\mathrm{obs}} \Bigg\},
\end{align*}
where we define $\pr{g}(\pr{x}, x \mid n, \alpha, \beta)\coloneqq g((n-\pr{x}+x)/2 \mid k, \alpha, \beta)$. The terminal conditions are
\begin{equation*}
    v(N, (N-k, x, a), (\alpha, \beta)) = \sum_{x^{\prime} \in \mathcal{S}^x_k} \pr{g}(\pr{x}, x \mid k, \alpha, \beta)\ r(x^{\prime}), \quad k<N.
\end{equation*}

For our experiment, we set the true value of $\theta = 0.3$ and a time horizon of $N = 50$. \Cref{bayesian_graph.jpg} illustrates a sample realisation of an optimal trajectory, given a prior of $\mathrm{Beta}(5,2)$, as well as the evolution of the estimate over $\theta$ over time.

\begin{figure}[t!]
\centering
  \includegraphics[width = 14cm]{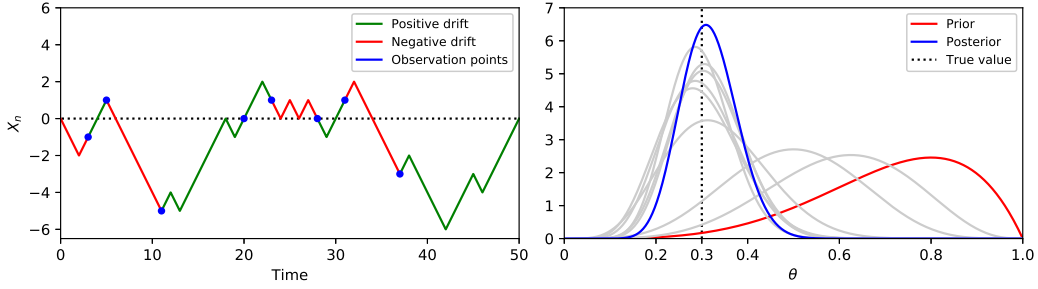}
  \caption{Left: sample realisation of the controlled random walk along the optimal trajectory. Right: prior and posterior distribution of $\theta$; the grey lines indicate `intermediate posteriors' obtained from earlier observations.}
  \label{bayesian_graph.jpg}
\end{figure}

We consider three different choices of $(\alpha, \beta)$ for the prior $\rho_0$ as well as varying the observation cost. For each parameter combination we compute the optimal policy and compare their respective performances across $5000$ sampled trajectories. A typical criteria of measuring the performance of the policy is to examine its regret, defined as
\begin{align*}
    \mathrm{reg}(N,\pi) =  J^{\theta^{*}}_0(\pi^{*}) - J^N(\pi),
\end{align*}
where $J^{\theta^{*}}_0(\pi^{*}) = \E^{\pi^{*}}[\sum r(x,a)]$ is the reward functional with no observation cost, with known parameter ${\theta^{*}}$ and optimal policy ${\pi^{*}}$, and $J^N(\pi)$ is the reward functional \eqref{reward_func_bayes} under a policy ${\pi \in \Pi_{\mathrm{obs}}}$. In this case we consider $\pi$ to be the optimal policy under observation costs and parameter uncertainty. The regret is therefore the cumulative sum of the suboptimal gap from the optimal policy. For the observation cost problem, the control between observations are constant and therefore suboptimal in general, as such we do not expect the regret to achieve asymptotically sublinear growth. Instead, we consider the following alternative criteria:
\begin{equation*}
    \mathrm{reg}_{c_{\mathrm{obs}}}(N, \pi) =  J^{\theta^{*}}_{c_{\mathrm{obs}}}(\pi^{*}) - J^N(\pi),
\end{equation*}
where here $J^{\theta^{*}}_{c_{\mathrm{obs}}}$ is the reward functional with known parameter $\theta^{*}$ and observation cost $c_{\mathrm{obs}}$, so that $\mathrm{reg}_{c_{\mathrm{obs}}}$ measures the contribution of the regret that arises from parameter uncertainty. On the left of \Cref{regret_graph}, we show the overall regret for varying the observation cost for a fixed prior, and on the right, $\mathrm{reg}_{c_{\mathrm{obs}}}$ is plotted with a fixed observation cost of $c_{\mathrm{obs}}=0.1$ across different initial priors $\rho_0$. As expected, the regret is generally higher when the prior estimate $\rho_0$ is less accurate, or when a larger $c_{\mathrm{obs}}$ value is used. Moreover the regret grows in a rather linear fashion. However, when examining the graph involving $\mathrm{reg}_{c_{\mathrm{obs}}}$ on the right side, we empirically observe sublinear growth. This can be interpreted as a gradual learning of the unknown parameters, despite the fact that observations only arrive in intervals. The results suggests that $\mathrm{reg}_{c_{\mathrm{obs}}}$ can be used as an alternative notion to capture the learning rate in problems involving observation costs, which we see as a possible direction for future analysis.

\begin{figure}[t!]
\centering
  \includegraphics[width = 14cm]{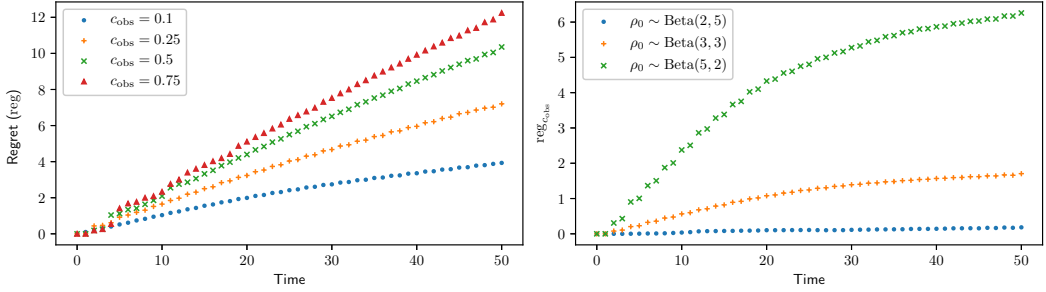}
  \caption{Left: regret over time for $\rho_0 \sim \mathrm{Beta}(3,3)$ for different values of $c_{\mathrm{obs}}$. Right: the growth of $\mathrm{reg}_{c_{\mathrm{obs}}}$ for fixed $c_{\mathrm{obs}}=0.1$ and different initial priors $\rho_0$.}
  \label{regret_graph}
\end{figure}

To demonstrate the effects of observation cost and prior estimates on the number of observations, we consider an alternative reward function, given by
\begin{align*}
r(x,a) = r(x) = \begin{cases}
2 & x = 0\\
-1 & x= \pm 2 \\
0 & \mathrm{otherwise}
\end{cases}
\end{align*}
In the absence of observation cost and parameter uncertainty, the controller aims to keep the process at the origin as often as possible, whilst avoiding the penalising boundary at $x \pm 2$. \Cref{bayesian_value_table} lists the performance of the optimal policies under each combination of observation cost and prior estimate. As the true value of $\theta = 0.3$, a prior of $\rho_0 \sim \mathrm{Beta}(2,5)$ acts a good estimate, and $\rho_0 \sim \mathrm{Beta}(5,2)$ acts as a poor estimate. In general, we see that the value of the observation cost $c_{\mathrm{obs}}$ has a more dominating effect on the resulting optimal policies and rewards obtained, as seen in the big drop-off in the number of observations when $c_{\mathrm{obs}} = 0.75$ in row (a), at which each observation comes at the cost of a significant proportion of the potential reward. Its effect on the sub-optimality is compounded with a bad prior estimate, with a negative reward and a 95\% credible interval width of 0.5 in the extreme case in the bottom-right entry of \Cref{bayesian_value_table}. 

\begin{table}[ht!]
    \centering
    \small
   \begin{tabular}{|c|c||c|c|c|c|c|}
   \hline
        & & $c_{\mathrm{obs}}= 0.1$ & $c_{\mathrm{obs}} = 0.25$ & $c_{\mathrm{obs}} = 0.5$ & $c_{\mathrm{obs}}=0.75$\\
        \hline
        & (a)  & 22.48 & 22.2 & 21.2 & 17.55 \\
        $\rho_0 \sim \mathrm{Beta}(2,5)$ & (b) & 20.622 & 17.15 & 11.26 & 6.0375 \\
        & (c)  & 0.2341 & 0.2360 & 0.2455 & 0.2844 \\
        \hline
        & (a)  & 21.4 & 20.97 & 18.36 & 11.27 \\
        $\rho_0 \sim \mathrm{Beta}(3,3)$ & (b) & 17.99 & 14.6475 & 8.92 & 2.5775 \\
        & (c)  & 0.2437 & 0.2459 & 0.2696 & 0.3624 \\
        \hline
        & (a)  & 19.22 & 17.3 & 11.21 & 3.34 \\
        $\rho_0 \sim \mathrm{Beta}(5,2)$ & (b) & 10.628 & 7.55 & 1.825 & -0.835 \\
        & (c)  & 0.2488 & 0.2583 & 0.3302 & 0.5034 \\
        \hline
   \end{tabular}
   \caption{Numerical results for the parameter uncertainty problem. Line (a): average number of observations. Line (b): average profit ($N=50$). Line (c): average credible interval width (HDI 95\%).}
    \label{bayesian_value_table}
\end{table}

\FloatBarrier

\section*{Acknowledgments.}
The authors would like to thank Prof. Dr. Dirk Becherer (Humboldt University of Berlin) for his insightful suggestions during discussion, as well as the two annonymous referrees for their feedback. Jonathan Tam is supported by the EPSRC Centre for Doctoral Training in Mathematics of Random Systems: Analysis, Modelling and Simulation (EP/S023925/1).

\bibliographystyle{abbrvurl}
\bibliography{ref}

\end{document}